\documentclass[10pt, a4paper]{amsart}
\usepackage[foot]{amsaddr}
\usepackage{header}
\usepackage{arydshln}
\usepackage{placeins}
\usepackage{subfig}
\usepackage{hyperref}

\makeatletter
\renewcommand{\email}[2][]{%
  \ifx\emails\@empty\relax\else{\g@addto@macro\emails{,\space}}\fi%
  \@ifnotempty{#1}{\g@addto@macro\emails{\mbox{\textrm{(#1)}}\space}}%
  \g@addto@macro\emails{#2}%
}
\makeatother

\begin{document}

 \title[Optimal dual isogeometric mortar methods]{Biorthogonal splines for optimal  weak patch-coupling in isogeometric analysis  with applications to finite deformation elasticity
}
\thanks{
Funds provided by the Deutsche Forschungsgemeinschaft under the contract/grant numbers: WO 671/11-1 as well as PO1883/1-1, WA1521/15-1 and  WO 671/15-1 (within the Priority Programme SPP 1748, "Reliable Simulation Techniques in Solid Mechanics. Development of Non-standard Discretisation Methods, Mechanical and Mathematical Analysis") are gratefully acknowledged.
}

\author{Linus Wunderlich}	 
\author{Alexander Seitz}	 
\author{Mert Deniz Alayd{\i }n}	 
\author{Barbara Wohlmuth}	 
\author{Alexander Popp} 

\email[L.~Wunderlich]{linus.wunderlich@ma.tum.de}
\email[A.~Seitz]{seitz@lnm.mw.tum.de}
\email[M.D.~Alaydin]{mert{\_}alaydin@brown.edu}
\email[B.~Wohlmuth]{wohlmuth@ma.tum.de}
\email[A.~Popp, Corresponding author]{alexander.popp@unibw.de}

\address[B.~Wohlmuth, L.~Wunderlich]{M2 - Zentrum Mathematik, Technische Universit\"at M\"unchen, Boltzmannstra\ss{}e~3, 85748~Garching, Germany }
\address[A.~Seitz]{Institute for Computational Mechanics, Technische Universit\"at M\"unchen, Boltzmannstra\ss{}e~15, 85748 Garching, Germany }
\address[M.D.~Alaydin]{School of Engineering, Brown University, 182 Hope Street, Providence, RI 02912, USA }
\address[A.~Popp]{Institute for Mathematics and Computer-Based Simulation, Universit\"at der Bundeswehr M\"unchen, Werner-Heisenberg-Weg 39, 85577 Neubiberg, Germany }
\date{}
\maketitle

\begin{abstract}
A new construction of biorthogonal splines for isogeometric mortar methods is proposed. The biorthogonal basis has a local support and, at the same time, optimal approximation properties, which yield optimal results with mortar methods. 
We first present the univariate construction, which has an inherent crosspoint modification. The multivariate construction is then based on a tensor product for weighted integrals, whereby the important properties are inherited from the univariate case.  
Numerical results including large deformations confirm the optimality of the newly constructed biorthogonal basis. 
\end{abstract}

\section{Introduction}
%!TEX root = ../article.tex
Weak patch-coupling is an important feature for practical applications of isogeometric analysis (IGA). With isogeometric methods, the computational domain is usually divided into several spline patches~\cite{cottrell:06,hughes:09,beirao:14,nguyen:15} and the solution to a partial differential equation is approximated by spline functions~\cite{Hoellig03} on each patch. Typically, multivariate splines are defined based on a tensor-product structure and a flexible coupling between the patches is important to gain some flexibility of the local meshes.
Different approaches are considered, e.g., 
Nitsche's method~\cite{nguyen:14,bletzinger:14}, penalty based methods~\cite{hofer:16} and mortar methods~\cite{dornisch:14,hesch:12}, and there is a recent interest in higher-order couplings, see, e.g.~\cite{coox:16,horger:18}. A recent practical review, which includes the related issue of trimming, is given in~\cite{marussig:17}.
Besides patch-coupling, theses methods are also used for the discretization of contact problems, see, 
e.g.,~\cite{delorenzis:14,antolin:17} and the references therein.

Mortar methods were originally applied in spectral and finite element methods~\cite{ben_belgacem:99,bernardi:94,wohlmuth:01}, and for the isogeometric case a mathematical stability and a priori analysis can be found  in~\cite{brivadis:15}.
The use of dual mortar methods~\cite{wohlmuth:00} yields computational advantages also for contact discretizations~\cite{popp:09,popp:10,popp:13,popp:14}. However, already for finite element methods the use of dual mortar methods for higher order methods poses additional difficulties, which could be solved by a local change of the primal basis~\cite{lamichhane:07,wohlmuth:12a,popp:12}.
The straightforward use of dual isogeometric mortar methods was considered in~\cite{seitz:16}, where a well-behavior for contact problems was observed, while for patch-coupling the convergence rate was severely reduced.
Here, we present a new construction of local dual basis functions with optimal approximation properties, based on the construction given in~\cite{oswald:02} for the finite element context. 
For the first time, this scheme allows to combine the two crucial features of local support of the dual basis and optimal approximation properties.
Alternative approaches include the use of basis functions that are not dual, but have a more convenient sparsity structure than standard basis functions~\cite{dornisch:17}. Very recently, it has been proposed in~\cite{zou:17} to refine one layer of elements along the interfaces to obtain a matching mesh. In the case of non-matching parametrizations, the mesh does not necessarily match even though the knots do match, so an unknown number of extra refinements is necessary. Also, the extension to a general three-dimensional setting remains unclear.

This article is structured as follows. In the next section, % Section~\ref{ch:problem_setting}
 we state the problem setting and briefly present standard isogeometric methods. 
 In Section~\ref{sec:dual_construction} we present the construction of the local dual basis functions with optimal approximation properties for one-dimensional and two-dimensional interfaces. 
The newly constructed dual basis functions are applied to isogeometric patch-coupling in Section~\ref{sec:numerical_results}
and the results of this work are summarized in Section~\ref{sec:conclusion}.

\section{Problem setting and recap of isogeometric mortar methods}
\label{ch:problem_setting}
%!TEX root = ../article.tex
Let $\Omega \subset \R^d$, $d=2,3$ be a bounded domain with a piecewise smooth boundary,  which is decomposed into two open sets $\GammaD, \GammaN$, such that $\GammaD\cap\GammaN=\emptyset$ and  $\barGammaD\cup\barGammaN = \partial \Omega$. 
The reference domain $\Omega$ with its points $\X$ is mapped at any instance of time $t$ to the deformed configuration $\Omega^{(t)}$ with its points $\x$ via the orientation preserving, invertible mapping $\motion_t:\Omega \to \Omega^{(t)}, \X \mapsto \x(\X)$, which defines the displacement field $\mathbf{u}=\x(\X)-\X$.
Starting from the deformation gradient $\defgrd=\idmat+\frac{\partial\mathbf{u}}{\partial\X}$, the right Cauchy--Green tensor $\rcg=\defgrd\vp\defgrd$ defines a non-linear measure of stretches. 
For simplicity, a hyperelastic material behavior is assumed, although the later presented mortar method directly applies to other constitutive relations as well.
For hyperelastic materials, the existence of a strain energy function $\sef$ is postulated, and the second Piola--Kirchhoff stress is then defined via $\PK=2\frac{\partial\sef}{\partial\rcg}$.
For homogeneous Dirichlet values, we
solve the quasi-static equilibrium equations of nonlinear elasticity on the domain $\Omega$:
\begin{align*}
\Divergence(\defgrd\PK) + \bodyforce &=\mathbf 0 ~  \text{ in } \Omega,\\
\mathbf u&=\mathbf 0  ~ \text{ on } \GammaD,\\
(\defgrd\PK)\refNormal&=\Neumanntraction ~ \text{ on } \GammaN,
\end{align*}
where $\refNormal$ denotes the outward unit-normal on $\GammaN$, $\bodyforce$ a body force vector per unit undeformed volume and $\Neumanntraction$ a given first Piola--Kirchhoff traction vector.

In the following, we briefly present the isogeometric mortar methods. 
For a more detailed presentation and the use of trace space Lagrange multipliers, see~\cite{brivadis:15}, for the extension to contact problems, see~\cite{seitz:16}.

\subsection{Standard spline spaces}
Let a spline degree $p$ and an open knot vector (i.e., first and last $p+1$ entries are repeated) $\Xi = (\xi_1,\ldots,\xi_{n+p+1})$ be given. The entries of $\Xi$ without their repetitions form the break point vector $Z = (\zeta_1,\ldots,\zeta_E)$, and $m_i$ denotes the multiplicity of $\zeta_i$ in $\Xi$. 
The Cox-de Boor recursion formula then defines the spline basis functions $\widehat{B}_i^p$, $i=1,
\ldots,n$ and the corresponding spline space $\widehat{S}^p(\Xi) = \spann_i \widehat{B}_i^p $ in the univariate setting. In the multivariate setting, consider $\boldsymbol \Xi = \Xi_1\times\cdots\times\Xi_d$ and the B-spline basis $\widehat{B}_{\mathbf i}^p(\boldsymbol \zeta) = \widehat{B}_{i_1}^p(\zeta_1)\cdots \widehat{B}_{i_d}^p(\zeta_d)$ with the spline space
$\widehat{S}^p(\boldsymbol \Xi) = \bigotimes_{\delta=1}^d \widehat{S}^p(\Xi_\delta) = \spann_{\mathbf  i}  \widehat{B}_{\mathbf i}^p $.
For simplicity of notation, we consider the same polynomial degree in all directions.

Introducing positive weights $w_{\mathbf i} > 0$ and the corresponding weight function $\NURBSWeight(\boldsymbol \zeta) = \sum_{\mathbf i\in\mathbf I} w_{\mathbf i} B_{\mathbf i}^p(\boldsymbol \zeta)$, we define the NURBS basis and space as
\[
\widehat N_{\mathbf i}^p(\boldsymbol \zeta) = \widehat B_{\mathbf i}^p(\boldsymbol \zeta) / \NURBSWeight(\boldsymbol \zeta), \quad 
\widehat N^p(\boldsymbol \Xi) = \{\widehat v_h =\widehat  w_h/\NURBSWeight, \quad\widehat w_h\in\widehat  S^p(\boldsymbol \Xi)\}.
\]

\subsection{Description of the computational domain}\label{subsec:computational_domain}
We consider a decomposition of the domain $\Omega$ into $K$ non-overlapping domains $\Omega_k$:
\[
\overline{\Omega} = \bigcup_{k=1}^K \overline{\Omega}_k, \text{ and } \Omega_i \cap \Omega_j = \emptyset  \text{ for } i \neq j.
\]
For $1\leq k_1, k_2 \leq K $, $k_1\neq k_2$, the interface is defined as the interior of the intersection of the boundaries, i.e., $\overline{\refInterface}_{k_1k_2} = \partial {\Omega}_{k_1} \cap \partial {\Omega}_{k_2}$, where ${\refInterface}_{k_1k_2}$ is open. The  non-empty interfaces  are enumerated as $\refInterface_l$, $l = 1,\,\ldots,\, L$. 
For each interface, one of the adjacent subdomains is chosen as the master side $m(l)$, the other one as the slave side $s(l)$, i.e., $\overline{\refInterface}_l=\partial \Omega_{m(l)}\cap \partial \Omega_{s(l)}$. The slave side is used to define the Lagrange multiplier space that enforces the coupling between the master and the slave side.   %On the interface $\refInterface_l$, we define the outward unit normal $\refNormal_l$ of the master side $\partial \Omega_{m(l)}$.

Each subdomain $\Omega_k$ is given as the image of the parametric space $\widehat \Omega = (0,1)^d$ by one single NURBS parametrization $\mathbf{F}_k\colon \widehat{\Omega} \rightarrow \Omega_k$,  $\mathbf{F}_k\in (N^p(\boldsymbol \Xi))^d$, which satisfies the regularity Assumption \cite[Assumption 1]{brivadis:15}: The parametrization $\mathbf{F}_k$ is a bi-Lipschitz homeomorphism, $\left.\mathbf{F}_k\right|_{\overline{ \bf Q}} \in C^{\infty}(\overline{\bf Q})$ and 
$\left.\mathbf{F}_k^{-1}\right|_{\overline{ \bf O}} \in C^{\infty}(\overline{\bf O})$
for any elements ${\overline{ \bf Q}}$ and ${\overline{ \bf O}}$ of the parametric  and the physical mesh, respectively.

Furthermore, we assume that the decomposition represents the Dirichlet boundary in the sense, that the pull-back of $\partial \Omega_k\cap\GammaD$ is either empty or the union of whole faces of the unit $d$-cube. 
We furthermore assume to be in a slave conforming situation, i.e.,  for each interface, the pull-back with respect to the slave domain is a whole face of the unit $d$-cube in the parametric space.
The $h$-refinement procedure yields a family of meshes, with each mesh being a uniform refinement of the initial one.

\subsection{The variational forms}
For each subdomain $\Omega_k$, we consider the local space $H^1_{\rm D}(\Omega_k) = \{v_h\in H^1(\Omega_k)\colon  \left. v_h \right|_{\partial \Omega_k \cap \GammaD} = 0\}$ and define the global broken Sobolev spaces $V= \Pi_{k=1}^K H^1_{\rm D}(\Omega_k)$ and $M= \Pi_{l=1}^L H^{-1/2}(\refInterface_l)$, endowed with the broken norms $ \| v \|_{V}^2 = \sum_{k=1}^K \| v \|_{H^1(\Omega_k)}^2$ and  $ \| v \|_{M}^2 = \sum_{l=1}^L \| v \|_{H^{-1/2}(\refInterface_l)}^2$.

Defining $\mathbf V = (V)^d$ and $\mathbf M = (M)^d$, we consider the broken non-linear form $a\colon\mathbf V\times\mathbf V\rightarrow \R$ and the   linear form $f_{\rm ext}\colon\mathbf V \rightarrow \R$:
\begin{gather*}
a(\mathbf u,\mathbf v) = \sum_{k=1}^K \int_{\Omega_k} 
% \stress(\mathbf u): \strain(\mathbf v) \dx
(\defgrd\PK):\Gradient \mathbf v \dx
, ~
f_{\rm ext}(\mathbf v) =\sum_{k=1}^K %\int_{\Omega_k} \widehat{\mathbf f}\vp\mathbf v\dx +
 \int_{\partial\Omega_k\cap \GammaN}\!\!\! \!\!\!\!\!\!\!  \Neumanntraction \vp \mathbf v  \dsurface
 +\int_{\Omega_k} \bodyforce\vp\mathbf v \dx.\end{gather*} 

\subsection{Isogeometric mortar discretization}
In the following, we define our discrete approximation spaces used in the mortar context, the mortar saddle point problem and the convergence order.
We introduce $V_{k,h}$ as the approximation space on $\Omega_k$ by \[V_{k,h}=\{v_k=\widehat{v}_k \circ \mathbf{F}_k^{-1} \in H_{\rm D}^1(\Omega_k), \widehat{v}_k \in\widehat N^{p_k}(\mathbf{\Xi}_k) \},\] which is defined on the knot vector $\mathbf{\Xi}_k$ of degree $p_k$, with $\mathbf V_{k,h}=(V_{k,h})^d$. 
On $\Omega$, we define the product space $\mathbf V_h = \Pi_{k=1}^K \mathbf V_{k,h} \subset \mathbf V$, which forms an $(H^1(\Omega))^d$ non-conforming space as it is discontinuous over the interfaces.

The mortar method is based on a weak enforcement of continuity across the interfaces $\refInterface_l$ in broken Sobolev spaces. Let a space of discrete Lagrange multipliers $M_{l,h}\subset L^2(\refInterface_l)$  on each interface $\refInterface_l$ be given. On the skeleton $\Gamma$, we define the discrete product Lagrange multiplier space $\mathbf M_h$ as $\mathbf M_h = \Pi_{l=1}^L M_{l,h}^d$.

One possibility for a mortar method is to specify the discrete weak formulation as a saddle point problem:
Find $ (\mathbf u_h, \LM_h) \in \mathbf V_h \times \mathbf M_h,$  such that 
\begin{subequations}\label{eq:IGM:discrete_spp}
\begin{align}
		a(\mathbf u_h, \mathbf v_h)+ b(\mathbf v_h, \LM_h) &= f_{\rm ext}(\mathbf v_h), \quad  \mathbf  v_h \in \mathbf V_h,\\
		b(\mathbf u_h, \LMtest_h) &= 0, \quad  \LMtest_h \in\mathbf  M_h,
\end{align}
\end{subequations}
where $b(\mathbf v,\boldsymbol \LMtest) = \sum_{l=1}^L \int_{\refInterface_l} \rho\, \boldsymbol \LMtest\vp [\mathbf v]_l \dsurface $ includes a weight $\rho\colon\refInterface_l\rightarrow \R$ and  $[\cdot]_l$ denotes the jump from the master to the slave side over $\refInterface_l$. The standard choice $\rho = 1$ will be altered in the three-dimensional case to simplify the construction of the Lagrange multiplier.

Due to the jump term, the coupling term decomposes in two integrals:
\[
\int_{\refInterface_l} \rho\, \boldsymbol \LMtest\vp [\mathbf v]_l \dsurface
= 
\int_{\refInterface_l} \rho\, \boldsymbol \LMtest\vp \mathbf v_s  \dsurface
-
\int_{\refInterface_l} \rho\, \boldsymbol \LMtest\vp \mathbf v_m \dsurface,
\]
where the second one includes the product of functions defined on the slave domain and the master domain on the interface.
As we assume the subdomains to match at the interface, the identity mapping on the geometric space suits as a projection between the spaces. In contrast, the isogeometric parametrizations of both subdomains are independent and may not match. 
To map a point $\boldsymbol \zeta_s$ in the parametric domain of the slave side to the equivalent point $\boldsymbol \zeta_m$ in the parametric domain of the master side, the inverse of the master geometry mapping is applied:
\[
\boldsymbol \zeta_m = \mathbf F_m^{-1} ( \mathbf F_s (\boldsymbol \zeta_s) ).
\]
We note that the accurate numerical integration of the coupling terms is important to obtain an optimal  method, see~\cite{brivadis:15b,wohlmuth:02,maday:97,farah:15}.

\subsection{Standard and biorthogonal  Lagrange multiplier spaces}

It is well-known from the theory of mixed and mortar methods, that two requirements guarantee the method to be well-posed and  of optimal order, see~\cite{ben_belgacem:99,wohlmuth:01,brezzi:13}. One is a uniform inf-sup stability of the discrete spaces and the second one an approximation requirement of the Lagrange multiplier.  
Given a sufficient approximation order and the inf-sup stability, then~\cite[Theorem 6]{brivadis:15} yields optimal order convergence rates.

Several stable trace spaces exist, but the structure of the resulting equation system~\eqref{eq:IGM:discrete_spp}, which is of a saddle point problem, causes a high computational effort. In comparison to a purely primal system, the saddle point system has more degrees of freedom, but also the solution of the indefinite equation system is more complicated, see the discussion in~\cite{benzi:05}. Without the use of biorthogonal basis functions, the reduction to a symmetric positive definite system in the primal variable involves the inversion of a non-diagonal mass matrix and severely disturbs the sparsity of the system.

For a clear insight, let us consider the block-structure of the saddle point problem arising within each Newton step for~\eqref{eq:IGM:discrete_spp} with a two-patch coupling, written in terms of the primal $u$ and dual $\lambda$ degrees of freedom:
\[
\begin{pmatrix}
K&M^\top\\
M&0
\end{pmatrix}
\begin{pmatrix}
\Delta u\\\lambda
\end{pmatrix}
=
-
\begin{pmatrix}
f\\0
\end{pmatrix},
\]
where $f=f_{\rm int} - f_{\rm ext}$.

 The saddle point problem is decomposed based on the degrees of freedom $u^{\rm S}_{\rm{I}}$ and $u^{\rm M}_{\rm{I}}$ belonging to the slave and the master body, respectively, except the ones on the interface, denoted $u^{\rm S}_{\Gamma}$ and $u^{\rm M}_{\Gamma}$, respectively: 
\begin{equation}\label{eq:full_spp}
\begin{pmatrix}
K^{\rm S}_{\rm{I}\rm{I}} & K^{\rm S}_{\rm{I}\Gamma}  & & & \\
K^{\rm S}_{\Gamma\rm{I}} & K^{\rm S}_{\Gamma\Gamma}  & & & M_{\rm{SS}}^\top \\
&& K^{\rm M}_{\rm{I}\rm{I}} & K^{\rm M}_{\rm{I}\Gamma}  & \\
&& K^{\rm M}_{\Gamma\rm{I}} & K^{\rm M}_{\Gamma\Gamma}  & -M_{\rm{SM}}^\top \\ 
& M_{\rm{SS}} &&-M_{\rm{SM}}&
\end{pmatrix}
\begin{pmatrix}
\Delta u^{\rm S}_{\rm{I}}\\
\Delta u^{\rm S}_{\Gamma}\\
\Delta u^{\rm M}_{\rm{I}}\\
\Delta u^{\rm M}_{\Gamma}\\
 \lambda
\end{pmatrix}
=
-
\begin{pmatrix}
f^{\rm S}_{\rm{I}}\\
f^{\rm S}_{\Gamma}\\
f^{\rm M}_{\rm{I}}\\
f^{\rm M}_{\Gamma}\\
 0
\end{pmatrix}.
\end{equation}
To reduce the saddle point problem, we note that with the mortar projection
$P=M_{\rm{SS}}^{-1}M_{\rm{SM}}$
 the last equation
yields
$\Delta u^{\rm S}_{\Gamma} = P \Delta u^{\rm M}_{\Gamma}$  
and the second equation yields
$ \lambda = - M_{\rm{SS}}^{-\top}\left(f^{\rm S}_\Gamma +  K^{\rm S}_{\Gamma\rm{I}}\Delta u^{\rm S}_{\rm{I}} + K^{\rm S}_{\Gamma\Gamma}P \Delta u^{\rm M}_{\Gamma} \right)$.
Then the saddle point problem reduces to the purely primal problem:
\begin{equation}\label{eq:spp_reduced}
\begin{pmatrix}
K^{\rm S}_{\rm{I}\rm{I}} &  &  K^{\rm S}_{\rm{I}\Gamma} P \\
& K^{\rm M}_{\rm{I}\rm{I}} & K^{\rm M}_{\rm{I}\Gamma}  \\
P^\top K^{\rm{S}}_{\Gamma\rm{I}}& K^{\rm M}_{\Gamma\rm{I}} & K^{\rm M}_{\Gamma\Gamma}  + P^\top K^{\rm S}_{\Gamma\Gamma}P
\end{pmatrix}
\begin{pmatrix}
\Delta u^{\rm S}_{\rm{I}}\\
\Delta u^{\rm M}_{\rm{I}}\\
\Delta u^{\rm M}_{\Gamma}
\end{pmatrix}
=
-
\begin{pmatrix}
f^{\rm S}_{\rm{I}} \\
f^{\rm M}_{\rm{I}}\\
f^{\rm M}_{\Gamma}+P^\top f^{\rm S}_\Gamma 
\end{pmatrix}.
\end{equation}
The sparsity of this matrix depends highly on the sparsity of the mortar projection $P$. Since $M_{\rm{SM}}$ is sparse, it depends on $M_{\rm{SS}}^{-1}$, which in general is dense, unless it is of diagonal form. A diagonal form is in general only achieved for biorthogonal basis functions.
\begin{figure}
\includegraphics[width=.32\textwidth]{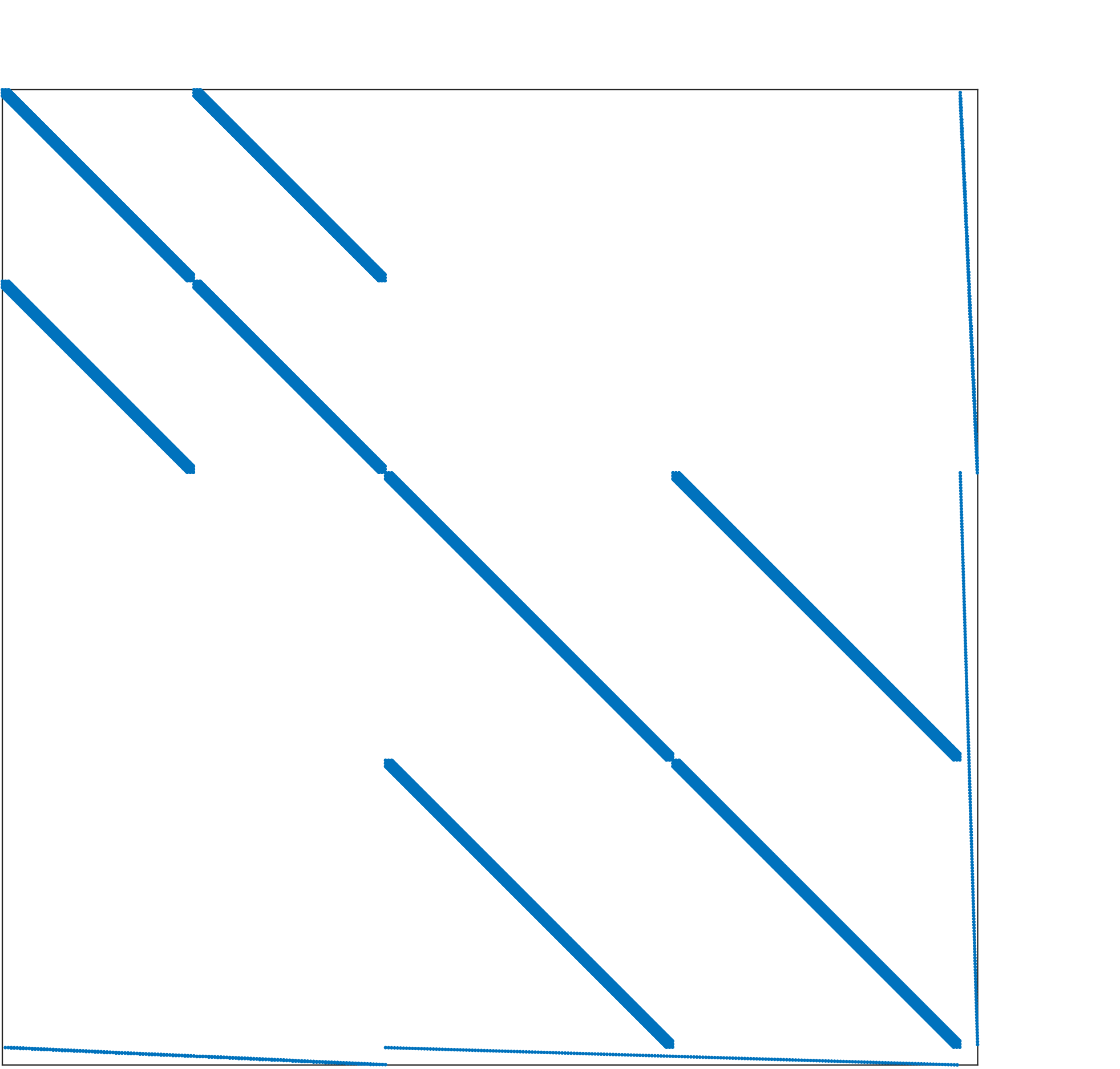}\hfill
\includegraphics[width=.32\textwidth]{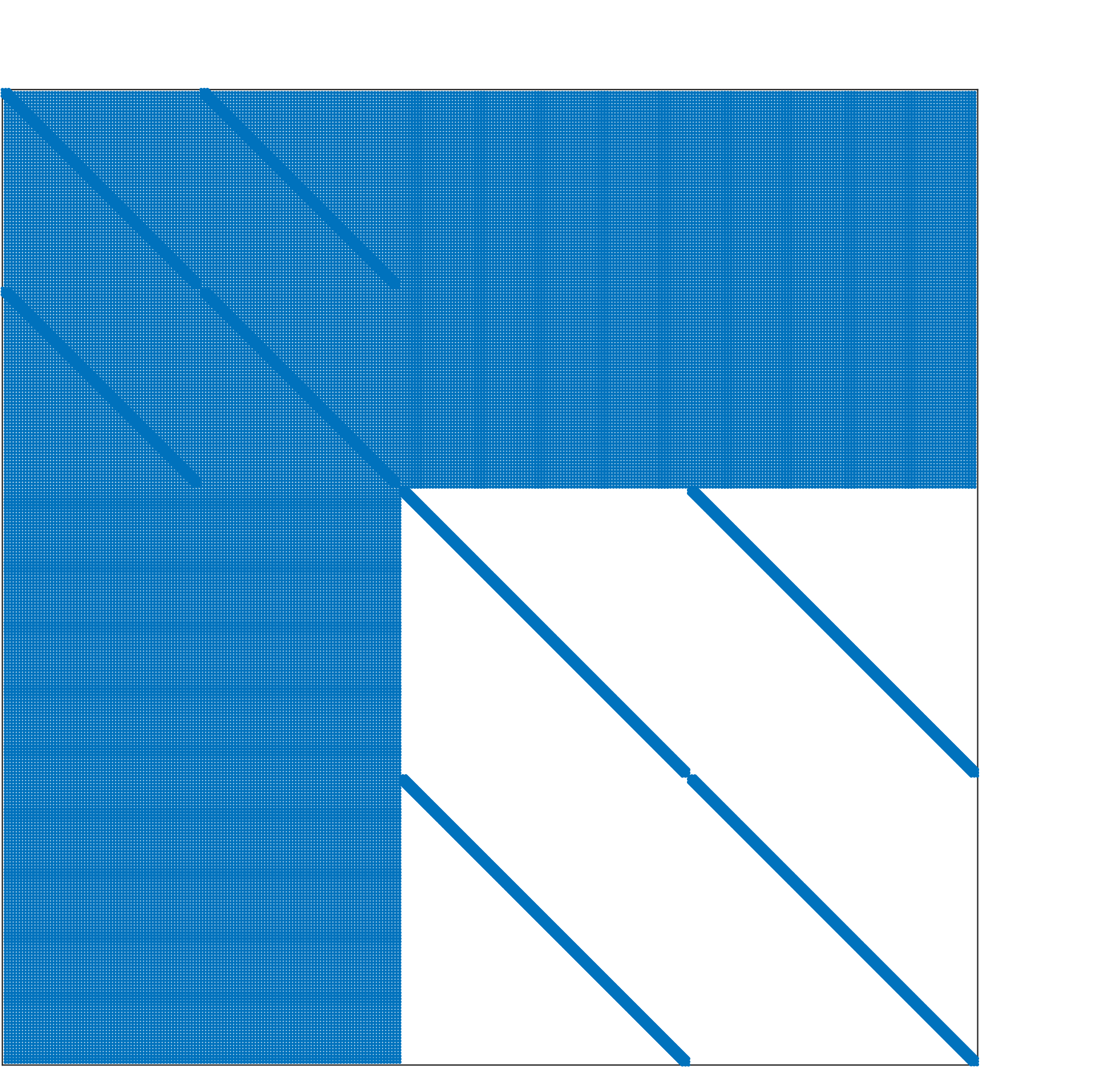}\hfill
\includegraphics[width=.32\textwidth]{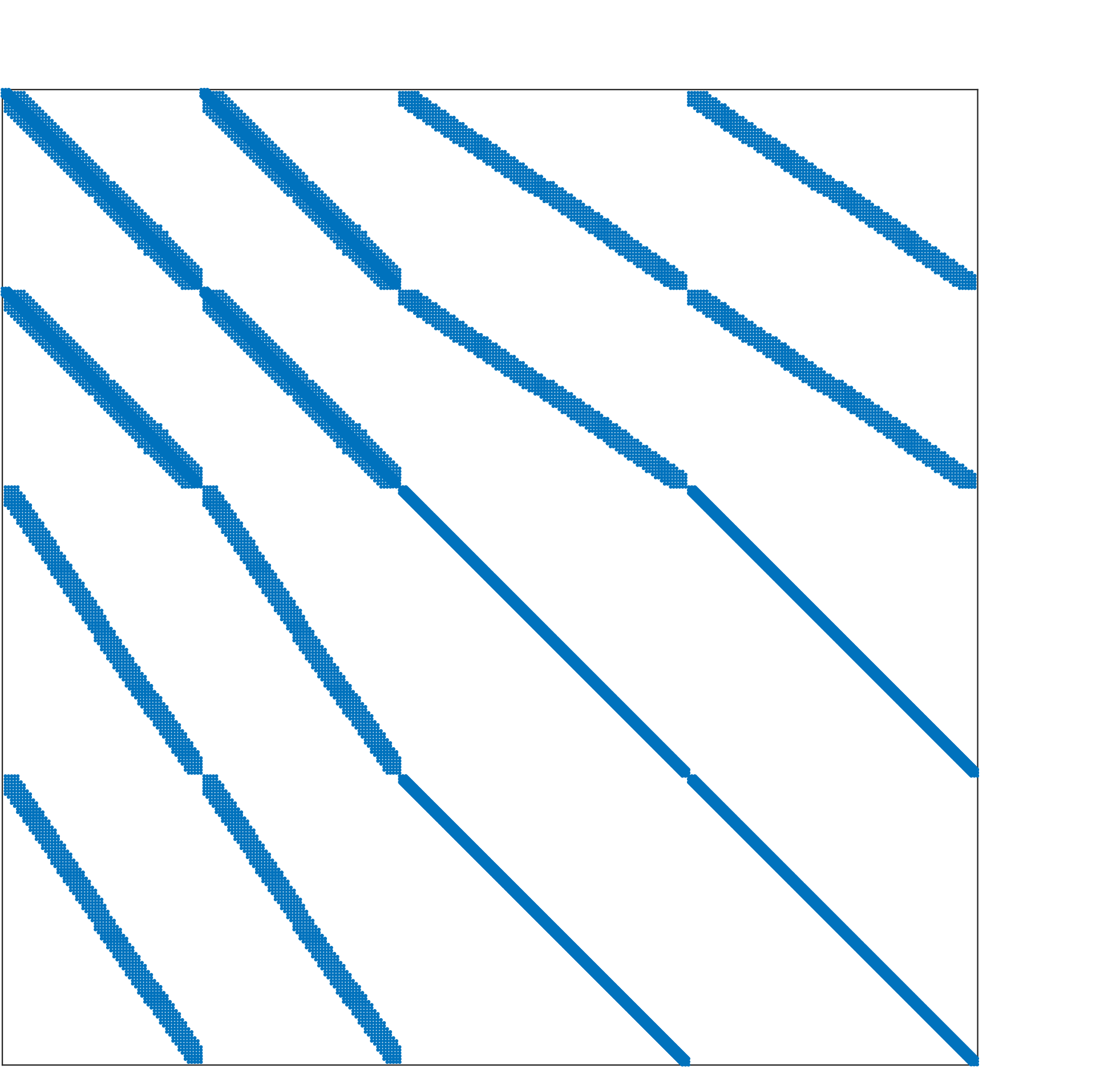}
\caption{Sparsity patterns for the equation system solved within each Newton step of the discrete non-linear system~\eqref{eq:IGM:discrete_spp}.  Two-dimensional, two-patch setting with $10\,560$ primal and $192$ dual degrees of freedom. Left: Saddle point structure~\eqref{eq:full_spp}. Middle: Reduced system~\eqref{eq:spp_reduced} for standard Lagrange multiplier. Right: Reduced system~\eqref{eq:spp_reduced} for dual Lagrange multiplier.}
\label{fig:sparsity_patterns}
\end{figure}
Examples of sparsity patterns for biorthogonal and standard mortar methods are shown in Figure~\ref{fig:sparsity_patterns}.

A first approach for dual basis functions in isogeometric mortar methods is presented in~\cite{seitz:16}, where a  local inversion of mass matrices is used to construct a dual isogeometric basis. This basis proved to be suitable for contact problems, but for mesh-tying problems, suboptimal convergence rates were observed, because the approximation order of the dual basis is insufficient. 
In the following, we will propose an alternative  construction of biorthogonal spline functions, which guarantees optimal approximation properties and is  based on the finite element construction in~\cite{oswald:02}.

\section{Construction of an optimal, locally supported biorthogonal basis}
\label{sec:dual_construction}
%!TEX root = ../article.tex
The presented construction is based on the same local polynomial spaces as the primal space. Solving local equation systems ensures optimality of the dual space as well as the locality of the resulting dual basis. 

With mortar methods, special care needs to be taken with crosspoints ($d=2$) or wirebaskets ($d=3$) to prevent an over-constrained global equation system:
\[
 \crosspoints= 
\left( \bigcup_{l\neq j}  \partial \refInterface_l\cap \partial\refInterface_j\right) \cup \left( \bigcup_l \partial\refInterface_l \cap \GammaD\right). 
\]
In the vicinity of crosspoints, we consider the restriction of the discrete primal space to zero boundary values: $H^1_{\rm CP}(\refInterface_l) = \{ v\in H^1(\refInterface_l), \left. v\right|_{\crosspoints} = 0\}$  and $W_{l,h} = \{ \left. v_h \right|_{\refInterface_l}, v_h \in V_{s(l), h}\}\cap H^1_{\rm CP}(\refInterface_l)$.

We consider the coupling $b(\mathbf v,\boldsymbol \LMtest) = \sum_{l=1}^L \int_{\refInterface_l} \rho\, \boldsymbol \LMtest\vp [\mathbf v]_l \dsurface$, for each interface $\Gamma_l$ individually and hence drop the index $l$. Furthermore, only the slave domain is considered, so we also drop the index $s(l)$.  
For each interface, we  transform the integral to the parametric domain: %For the NURBS $\mathbf v_h = 
For a  NURBS basis function $N_{\mathbf i} = (\widehat{{B}}_\mathbf i  \circ \mathbf F^{-1} ) / (\NURBSWeight \circ \mathbf F^{-1})$  and a Lagrange multiplier basis $\psi_\mathbf j = \widehat \psi_\mathbf j  \circ \mathbf F^{-1}$, we get
\begin{align}
\int_{\refInterface} \rho\,\left(N_\mathbf i\mathbf e_k\right) \vp\left( \psi_\mathbf j \mathbf e_{k'}\right)  \dsurface
 &=  \delta_{kk'} \int_{\refInterface} \rho\, N_\mathbf i \psi_\mathbf j \dsurface \notag\\
&= \delta_{kk'}\int_{\gammahat}\widehat \rho\,  \widehat B_\mathbf i^{p} \widehat\psi_\mathbf j  \dsurfacehat,\label{eq:integral_trafo}
\end{align}
with $\widehat \rho = \rho \circ \mathbf{F}~ \det{\nabla_\gammahat \mathbf{F}}/\NURBSWeight$.
We use the Kronecker delta $\delta_{ij}$, which equals one if $i=j$ and zero otherwise.  
Biorthogonality is characterized by the relation 
\[
 \int_{\gammahat}\widehat \rho\,  \widehat B_\mathbf i^{p} \widehat\psi_\mathbf j  \dsurfacehat = c_{\mathbf i} \, \delta_{\mathbf i \mathbf j}.
\]

We start with the one-dimensional construction and later  consider the tensor-product extension based on the use of weighted but equivalent $L^2$-spaces. 
\subsection{Unilateral construction}
The biorthogonal basis with polynomial reproduction is constructed by carefully studying the required properties. It is defined within a broken space of local polynomials of the same degree, where a family of  biorthogonal basis functions with local support exists. Then several local equation systems are solved to find a basis with local support and the desired polynomial reproduction. 

 Without loss of generality $\refInterface = \mathbf{F}((0,1)\times\{0\})$, i.e., $\gammahat = (0,1)\times\{0\}$.
We consider the weighted $L^2$-product $(u,v)_{\Normsub} = \int_0^1 u\,v\,\widehat{\rho} \dsurfacehat$ as given by~\eqref{eq:integral_trafo}.
The trace space of B-splines on the parametric space is then given by: 
\[W_h = \left\{\widehat v_h = \left. v_h\circ \mathbf{F} ~ \NURBSWeight\right|_{
\gammahat} , v_h\in V_{h} \cap H_{\rm CP}^1(\refInterface)   \right\}  \subset\widehat S^p(\boldsymbol \Xi_1),\]
with the basis $\widehat{B}^p_i$, $i\in\Ibsp=\{\imin,\ldots,\imax\}$, where $\imin\in\{1,2\}$ and $\imax\in\{n-1,n\}$, depending on  $\crosspoints$. 

The biorthogonal basis is constructed within the broken polynomial space of the same degree as the spline space:
\[
W_h^{-1} = \{ v\in L^2(0,1), \left.v\right|_{e_i}\in \Poly_p\},
\]
which is of dimension $N =  \dim W_h^{-1} = (E-1)(p+1)$, where $E-1$ is the number of elements on $\gammahat$ and $e_i = (\zeta_i, \zeta_{i+1})$. 
Since $W_h \subset W_h^{-1}$, we can extend the B-spline basis to a basis of the broken space. A convenient basis with the desired support  is constructed in the following.

The support of the extended basis is desired not to be larger than the support of a single B-spline function.
This is ensured by decomposing the broken polynomial space $W_h^{-1}$  into $n$~subspaces $W_{h,i}^{-1} $ by breaking apart the B-spline basis:
\[
W_{h,i}^{-1} = \spann \left\{\left.\widehat{B}^p_i\right|_{e_j},~~j=1,\ldots E-1\right\}, \quad i=1,\ldots,n.
\]
Since each basis function is supported on at most $p+1$ elements, it holds $n_i = \dim W_{h,i}^{-1}\leq p+1$, and, since B-splines form a local polynomial basis, it indeed holds that $W_h^{-1} = \bigotimes_{i=1}^n W_{h,i}^{-1}$.
Within each local space $W_{h,i}^{-1}$  we extend  $\widehat{B}^p_i$ to a basis, i.e., we define $\phi_{i,j}\in W_{h,i}^{-1}$, $j=1,\ldots,n_i-1$, such that 
\[
W_{h,i}^{-1} = \spann \left( \widehat{B}^p_i, (\phi_{i,j})_{j=1,\ldots,n_i-1} \right).
\]
Then, the local basis functions are combined to a basis of $W_h^{-1}$:
\[\big(\varphi_i\big)_{i=1,\ldots,N} = 
\big(\widehat{B}^p_1,\ldots,\widehat{B}^p_n, (\phi_{1,j})_{j=1,\ldots,n_1-1}, \ldots,( \phi_{n,j})_{j=1,\ldots,n_n-1} \big).
\]
Any choice of the local basis functions yields the desired support, but to simplify the algebraic construction of a biorthogonal basis, we consider the construction presented in the following Remark~\ref{rem:construction_local_basis}.

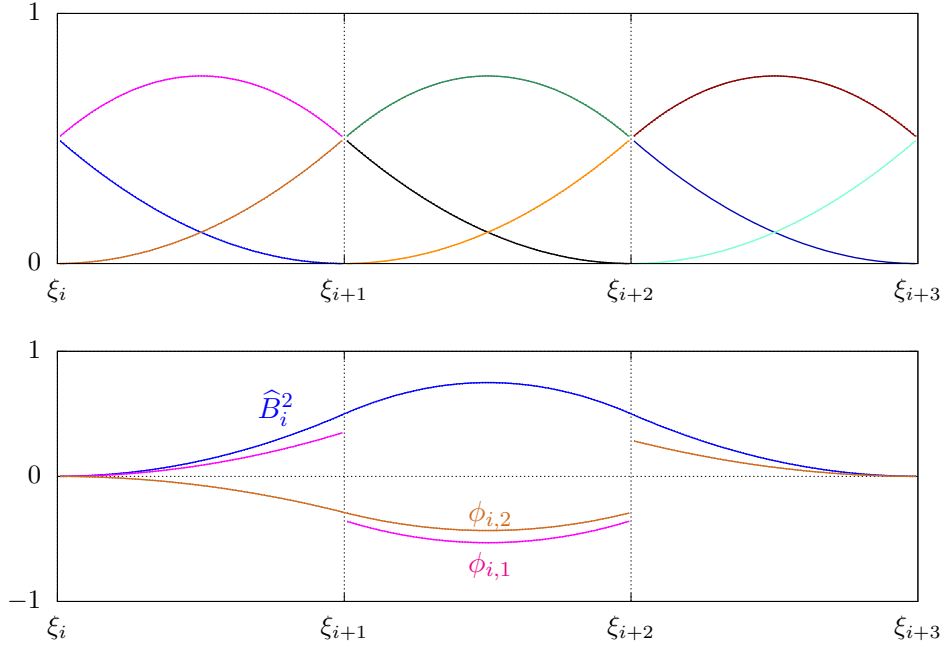
\begin{figure}

\hbox{
\setlength{\unitlength}{0.0500bp}%
\begin{picture}(7200,2520)(0,0)%
  \put(6839,200){\makebox(0,0){\strut{}$\xi_{i+3}$}}%
  \put(4699,200){\makebox(0,0){\strut{}$\xi_{i+2}$}}%
  \put(2560,200){\makebox(0,0){\strut{}$\xi_{i+1}$}}%
  \put(420,200){\makebox(0,0){\strut{}$\xi_{i}$}}%
  \put(300,2280){\makebox(0,0)[r]{\strut{}$1$}}%
  \put(300,400){\makebox(0,0)[r]{\strut{}$0$}}%
   \put(0,0){\includegraphics[scale=1]{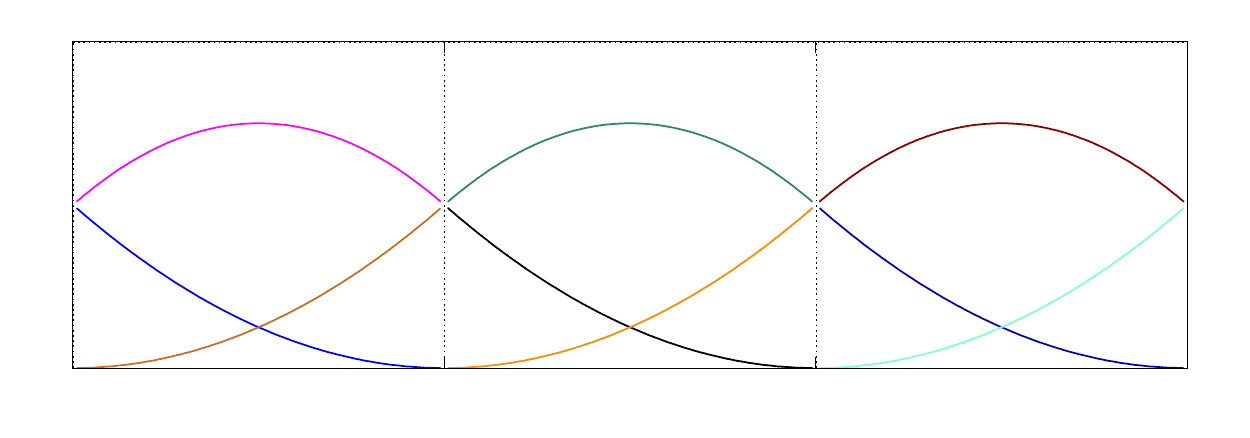}}
\end{picture}%
}
\hbox{
\setlength{\unitlength}{0.0500bp}%
\begin{picture}(7200,2520)(0,0)%
  \put(6839,200){\makebox(0,0){\strut{}$\xi_{i+3}$}}%
  \put(4699,200){\makebox(0,0){\strut{}$\xi_{i+2}$}}%
  \put(2560,200){\makebox(0,0){\strut{}$\xi_{i+1}$}}%
  \put(420,200){\makebox(0,0){\strut{}$\xi_{i}$}}%
  \put(300,2280){\makebox(0,0)[r]{\strut{}$1$}}%
  \put(300,1340){\makebox(0,0)[r]{\strut{}$0$}}%
  \put(300,400){\makebox(0,0)[r]{\strut{}$-1$}}%
   \put(0,0){\includegraphics[scale=1]{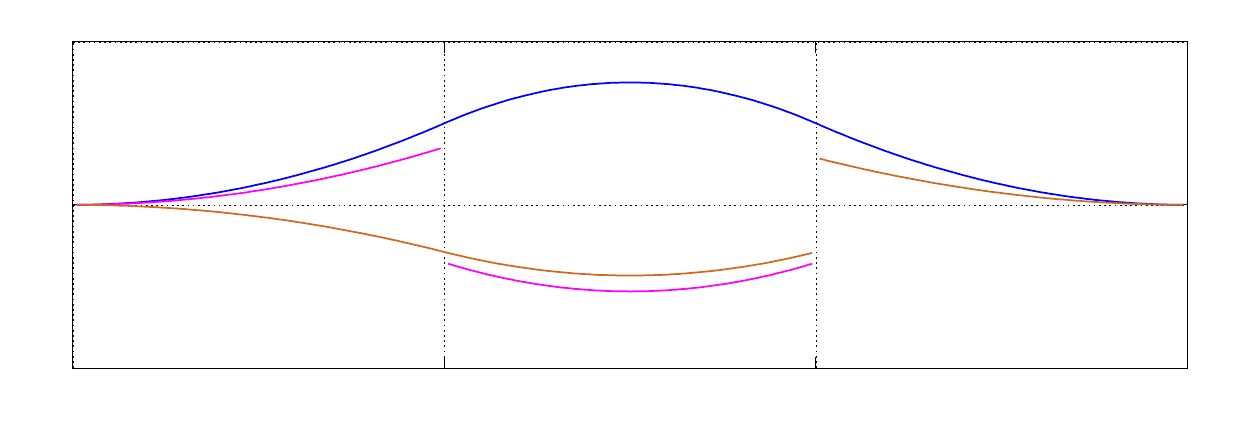}}
  \put(2050,1850){\makebox(0,0){\strut{}\large\color{blue}$\widehat{B}^2_i$}}%
  \put(3650,670){\makebox(0,0){\strut{}\large\color{magenta}$\phi_{i,1}$}}%
  \put(3650,1050){\makebox(0,0){\strut{}\large\color[rgb]{0.82,.41,.11}$\phi_{i,2}$}}%
\end{picture}%
}
  \caption{Illustration of the broken basis. Top: broken basis functions $\widehat B^2_{i,j}$. Bottom: Local basis $(\widehat{B}_i^2, \phi_{i,1}, \phi_{i,2})$, based on the orthogonal basis $A_j$ (normalized).}
  \label{fig:broken_functions}
\end{figure}

\begin{remark}[Construction of the local basis functions $\phi_{i,j}$]\label{rem:construction_local_basis}
We describe the construction for the first $\lceil n/2\rceil$ basis functions and note that the last basis functions can be constructed analogously, by transforming the index as $\widetilde\imath = n+1-i$.

Consider  $\widehat{k}$, such that $\supp \widehat{B}_i^p = [\zeta_{\widehat{k}}, \zeta_{\widehat{k}+n_i}]$. Then restricting $\widehat{B}_i^p$ to each each element of the support yields
\[\widehat{B}^p_{i,j} = \left. \widehat{B}_i^p\right|_{e_{\widehat{k}+j-1}}, \quad j=1,\ldots,n_i,\]
which decomposes $\widehat{B}_i^p$: 
\begin{equation}\label{eq:decomposition_B_spline}
\widehat{B}_i^p = \sum_{j=1}^{n_i} \widehat{B}_{i,j}^p,
\end{equation}
as illustrated in the top picture of  Figure~\ref{fig:broken_functions}.

Based on these restrictions, we extend the B-spline $\widehat{B}_i^p$ to a basis of $W_{h,i}^{-1}$ by defining
\[
\phi_{i,j} = \sum_{k=1}^{n_i} \alpha_{jk} \widehat{B}_{i,k}^p, \quad j=1,\ldots, n_i-1
\]
for an orthogonal basis $A_j = (\alpha_{jk})_k$, $j=0,\ldots,n_i-1$ of $\R^{n_i}$. The orthogonality is beneficial for the algebraic construction of a biorthogonal basis as it requires less matrix multiplications. Our choice  of $A_j$ is specified in the following.

We set $\widetilde{A}_0 = (1,\ldots,1) \in \R^{n_i}$, which corresponds to the decomposition~\eqref{eq:decomposition_B_spline} of $\widehat B_i^p$, and extend it for $j=1,\ldots,n_i-1$:
\[
\widetilde{A}_j = (\underbrace{-1,\ldots,-1}_{j~\rm{times}}, ~j, ~0,\ldots) \in \R^{n_i}.
\]
The resulting basis is of the following form:
\begin{equation*}
\begin{matrix}
\widetilde A_0:\\
\widetilde A_1:\\
\widetilde A_2:\\
\widetilde A_3:\\
\widetilde A_4:\\
\vdots
\end{matrix}
\hspace{2em}
\begin{matrix}
\vphantom{\widetilde A_0} \phantom{-}1	& \phantom{-}1 	& \phantom{-}1 	& \phantom{-}1  &\phantom{-}1& \ldots\\\hdashline
\vphantom{\widetilde A_1} -1 		&  \phantom{-}1	&  	&   && \\
\vphantom{\widetilde A_2} -1 		& -1	& \phantom{-}2	&  & &\\
\vphantom{\widetilde A_3} -1 		& -1	& -1	& \phantom{-}3 & \\
\vphantom{\widetilde A_4} -1 		& -1	& -1	& -1 & \phantom{-}4 & \\
\vdots 	& \vdots	&\vdots	&\vdots	&	& \ddots
\end{matrix}
\end{equation*}
For stability of the basis, we prefer a rather symmetric definition of the basis, so we permute the basis as follows (presented for $n_i$ even):
\begin{equation*}
\pi = \left( \frac{n_i}{2}, ~\frac{n_i}{2}-1, ~\frac{n_i}{2} + 1, \ldots, ~1,~ n_i\right),
\end{equation*}
i.e. $\alpha_{j\pi(k)} = {\widetilde\alpha}_{jk}$.
This yields a pyramid-like structure of the constructed indices~$\alpha_{jk}$:
\begin{equation*}
\begin{matrix}
A_0:\\
A_1:\\
A_2:\\
A_3:\\
A_4:\\
\vdots
\end{matrix}
\hspace{2em}
\begin{matrix}
\vphantom{A_0} 1&\ldots & 1&1 &1 &1 &1 & \ldots & 1\\
\hdashline
\vphantom{ A_1} &	&	&	 \phantom{-}1& -1	&	&	&\\
\vphantom{ A_2} &	&	&	-1& -1	& \phantom{-}2	&	&\\
\vphantom{ A_3} &	&	 \phantom{-}3&	-1& -1	&-1	&	&\\
\vphantom{ A_4} &	&	-1&	-1& -1	&	-1& \phantom{-}4	&\\
&\reflectbox{$\ddots$}	&	&	& \vdots	&	&	& \ddots
\end{matrix}
\end{equation*}
The constructed basis functions $\phi_{i,j}$ are  illustrated at the bottom of  Figure~\ref{fig:broken_functions}. 
\end{remark}

Similar to the standard construction of dual basis functions~\cite{seitz:16}, 
we can construct $\widetilde \psi_i\in W_h^{-1}$, $i=1,\ldots,N$, as the biorthogonal basis to $(\varphi_i)_{i=1,\ldots,N}$:
\[
\left( \widetilde \psi_i, \varphi_j \right)_\Normsub = \delta_{ij}.
\]
More precisely, the biorthogonal basis is defined element by element. On each element, the local mass matrix is computed and inverted, yielding a local biorthogonal basis, which defines $\widetilde\psi_i$ on this element. 

We point out that the basis $(\varphi_i)_{i=1,\ldots,N}$ can be separated into the primal basis, which spans $W_h$ for $i\in \Ibsp$,  and the remaining functions enhancing the basis to span $W_h^{-1}$: $i\in \Iextra = \{1,\ldots,N\}\backslash \Ibsp$.

Based on this basis, a family of biorthogonal basis functions to the B-splines can be constructed:
\[
 \widetilde \psi_i + \sum_{k\in\Iextra} z_{ki} \widetilde \psi_k
\]
for any $z_{ki}$, $k\in \Iextra$, $i\in\Ibsp$.  Since the $\widetilde \psi_k$ have a local support, a suitable sparse choice of the $z_{ki}$ yields a local basis. The choice of the non-zero values is determined by local equation systems, which finish the construction.

For each $\hat k\in\Iextra$, let us choose  an index set $\Iloc(\hat k)\subset\Ibsp$ with $\left|\Iloc(\hat k)\right| = p+1$. 
The choice of the index set is discussed in the following Remark~\ref{rem:index_sets}.
Let $(p_i)_{i=1,\ldots,p+1}$ be a basis of $\Poly_p$,  e.g., monomials or a set of orthogonal polynomials. 
 Then we set $z_{\hat k i} = 0$ for $i\not \in \Iloc(\hat k)$ and solve the following square linear equation system for the remaining entries $z_{\hat k i}$, $i\in  \Iloc(\hat k)$:  
\begin{equation}\label{eq:square_local_system}  
 \sum_{i\in\Iloc(\hat k)} \left(p_l,  \widehat{B}^p_i\right)_\Normsub  z_{\hat k i} = 
 (p_l, \varphi_{\hat k})_\Normsub, \quad l = 1,\ldots,p+1.
\end{equation} %alternativ: Matrixschreibweise, ist aber kompliziert, weil dann durchindizierung von I(j) notwendig ist. 

The following diagram sketches how the sparsity structure of $z_{ki}$ depends on $\Iloc(\hat k)$ for $p=2$:
\begin{align*}
&\quad i\in \Ibsp \rightarrow \\
(z_{k,i})_{k,i} : \qquad 
\begin{matrix}
k\in\Iextra\\\downarrow \\  \\ \\
 \hdashline \hbox{\scriptsize $\hat k$} \\
 \hdashline  \\
\end{matrix}&
\begin{bmatrix}
 * & * & * & & & & & & & \\
 * & * & * & & &  & & & & \\ 
   & * & *  & * & &  & &  & & \\
   &   &	&	& \makebox(0,0)[cb]{$\ddots$} &  & &	& & \\
 \hdashline
   &	&   &	& 	&	* & * & *	& &\\
 \hdashline
 &	&	&	& 	& &\makebox(0,0)[cb]{\vspace{-1em}$\underbrace{\hspace{3em}}_{i\in\Iloc(\hat k)} $}&&	&\\
% &	&	&	& 	&\makebox(0em,0em)[ct]{}&&	&\\
% &	&	&	& 	&\makebox[0em]{\scriptsize $\Iloc(\hat k)$}&&	&\\
\end{bmatrix}\vspace{1em}
\end{align*}

\begin{remark}[Choice of the index sets]\label{rem:index_sets}
Let $\supp\widetilde{\psi}_{\hat k} = (\zeta_{k_l}, \zeta_{k_r})$, consider the central element $\widehat{e} = e_{\lfloor(k_l+k_r)/2\rfloor}$.  Then let $\Iloc(\hat k)$ contain the $p+1$ indices of the B-splines that are supported on the element $\widehat e$:
\begin{equation*}
\Iloc(\hat k) =
\left\{i\in \Ibsp, \widehat{e} \subset \supp \widehat{B}_{i}^p\right\} .
\end{equation*}
\end{remark}
The biorthogonal basis with local support and optimal approximation order~$p$ is then defined as
\begin{equation} 
\label{eq:definition_optimal_biorthogonal}
\widehat \psi_i = \widetilde \psi_i + \sum_{k\in\Iextra} z_{ki} \widetilde \psi_k, \quad \text{ for } i\in\Ibsp.
\end{equation}
The support of $\widehat  \psi_i$ is determined by the choice of the index sets.
Since  $z_{ki}\neq 0$  yields $\supp \widetilde \psi_k \subset \supp \widehat  \psi_i$, we can estimate the support of $\widehat  \psi_i$ by 
\[
 \supp \widehat  \psi_i
 \subset
 \supp \widehat{B}_i^p \cup 
 \bigcup_{k\in\Iloc(i)}
\supp \widetilde \psi_k.
\]
By construction,  the support of $\widetilde \psi_j$ overlaps partially with the support of $\widehat{B}_i^p$. Since it contains at most $p+1$ elements, the support for the presented construction contains at most $2p+1$ elements.

The following Theorem~\ref{thm:optimality_proof} concludes this section by proving the optimality of the constructed biorthogonal basis.

\begin{theorem}\label{thm:optimality_proof}
Assume that for all $\hat k\in\Iextra$, the equation system~\eqref{eq:square_local_system} has full rank. Then the basis $(\widehat \psi_i)_{i\in\Ibsp}$ defined by~\eqref{eq:definition_optimal_biorthogonal} is an optimal biorthogonal basis, i.e., it  fulfills  biorthogonality
\[
\left(\widehat{B}^p_i, \widehat \psi_j\right)_{\widehat \rho} = \delta_{ij},
\]
for $i,j\in\Ibsp$ 
and optimal convergence on $M_h = \spann_i \widehat \psi_i$, i.e.,  for any $\lambda\in H^{p+1}(\refInterface)$
\[
\inf_{\mu_h\in M_h} \| \mu_h-\lambda\|_{L^2(\refInterface)} \leq c h^{p+1} \|\lambda\|_{H^{p+1}(\refInterface)}.  
\]
\end{theorem}
\begin{proof}
The proof follows the ideas of the finite element case, see~\cite{oswald:02}.
By construction of $\left( \widetilde \psi_i\right)_{i=1,\ldots,N}$, for any choice of $(z_{ki})_{k,i}$,
\begin{equation*} 
 \widetilde \psi_i + \sum_{k\in\Iextra} z_{ki} \widetilde \psi_k, \quad \text{ for } i\in\Ibsp,
\end{equation*}
is a biorthogonal basis to $\left(\widehat{B}^p_i\right)_{i\in\Ibsp}$.

Now, let us show that the choice of  $z_{ki}$ for $k\in\Iextra$ and $i=\Ibsp$ guarantees  polynomial reconstruction. 
Therefore, we show that the quasi-interpolation 
\begin{equation*} 
\operatorname{\mathcal{Q}}f = \sum_{i\in\Ibsp} \left(f,  \widehat{B}^p_i\right)_{\Normsub} \widehat \psi_i
\end{equation*} 
is invariant for polynomials $\Poly_p$, which is equivalent to
\[
(p_l, \varphi_j)_{\Normsub} = \sum_{i\in\Ibsp} \left(p_l,  \widehat{B}^p_i\right)_{\Normsub} \left( \widehat \psi_i, \varphi_j\right)_{\Normsub},
\quad \text{ for any } j=1,\ldots,N,~l=1,\ldots p+1. 
\]

For $j\in\Ibsp$ it holds that $\varphi_j = \widehat{B}^p_j$, and  we can directly use the biorthogonality of $\widehat \psi_i$ and $\varphi_j$:
\[
 \sum_{i\in\Ibsp} \left(p_l,  \widehat{B}^p_i\right)_{\Normsub} \left( \widehat  \psi_i, \varphi_j\right)_{\Normsub}
 = 
 \sum_{i\in\Ibsp} \left(p_l,  \widehat{B}^p_i\right)_{\Normsub} \delta_{ij} =  \left(p_l,  \varphi_j\right)_{\Normsub}. 
 \]

For $j\in\Iextra$, biorthogonality cannot be directly used, but the considered  equation system~\eqref{eq:square_local_system} yields, since $\Ibsp\cap\Iextra=\emptyset$ and $z_{ji}= 0$ for $i\not\in\Iloc(j)$:
\begin{align*}
 \sum_{i\in\Ibsp} \left(p_l,  \widehat{B}^p_i\right)_{\Normsub} \left( \widehat \psi_i, \varphi_j\right)_{\widehat \rho} &=
 \sum_{i\in\Ibsp} \left(p_l,  \widehat{B}^p_i\right)_{\Normsub} \left( \widetilde \psi_i + \sum_{k\in\Iextra} z_{ki} \widetilde \psi_k , \varphi_j\right)_{\Normsub}
\\&\hspace{-1em} = 
 \sum_{i\in\Ibsp} \left(p_l,  \widehat{B}^p_i\right)_\Normsub  z_{ji} =
 \sum_{i\in\Iloc(j)} \left(p_l,  \widehat{B}^p_i\right)_\Normsub  z_{j i} = 
 (p_l, \varphi_{j})_\Normsub .
\end{align*}
\end{proof}

At the end of this construction, the biorthogonal basis functions can be scaled as desired. A common scaling, e.g.~\cite{seitz:16}, is
\[
\int_\refInterface  N_\mathbf{i} \psi_\mathbf{j} \dsurface
= 
\delta_{\mathbf {i} \mathbf{j}}
\,
\int_\refInterface  N_\mathbf i \dsurface
\]
The newly constructed biorthogonal basis functions are shown in Figure~\ref{fig:basis_functions}, where they are compared to the naive biorthogonal basis functions from \cite{seitz:16} and the primal basis functions. 
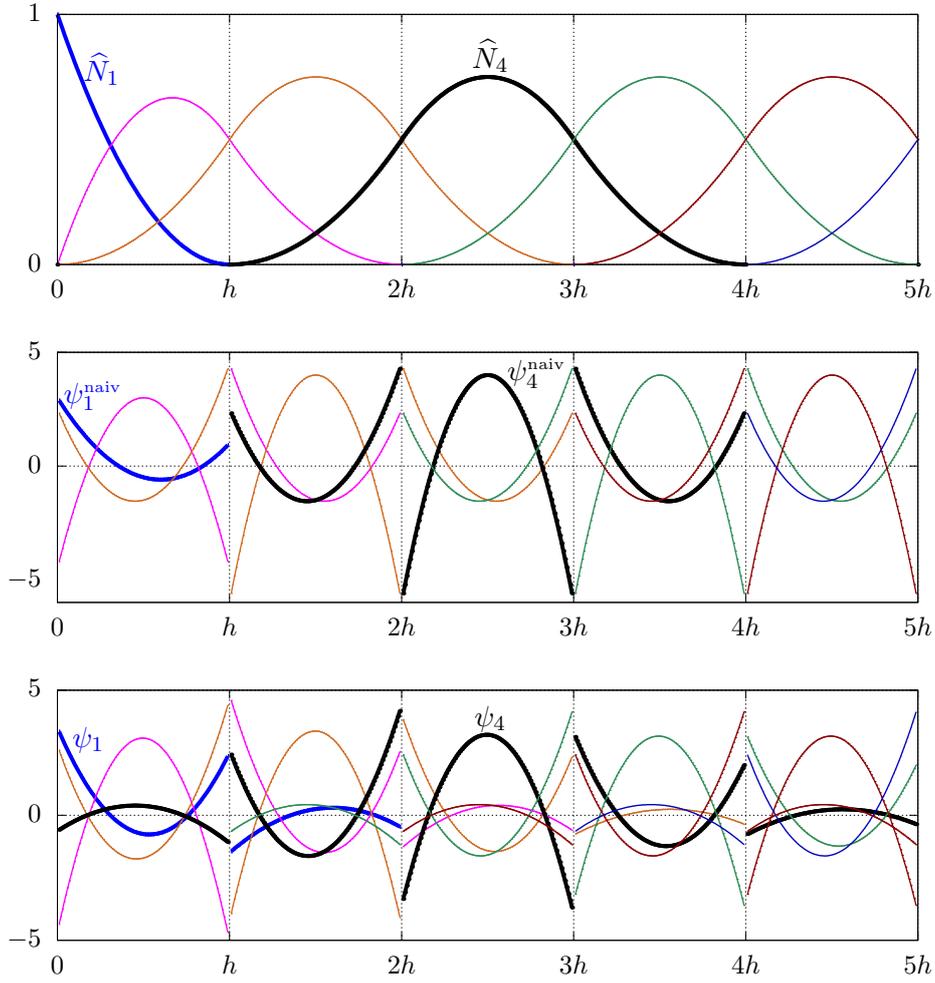
\begin{figure}
\hbox{
\setlength{\unitlength}{0.0500bp}%
\begin{picture}(7200,2520)(0,0)%
  \put(6839,200){\makebox(0,0){\strut{}$5h$}}%
  \put(5555,200){\makebox(0,0){\strut{}$4h$}}%
  \put(4271,200){\makebox(0,0){\strut{}$3h$}}%
  \put(2988,200){\makebox(0,0){\strut{}$2h$}}%
  \put(1704,200){\makebox(0,0){\strut{}$h$}}%
  \put(420,200){\makebox(0,0){\strut{}$0$}}%
  \put(300,400){\makebox(0,0)[r]{\strut{}$0$}}%
  \put(300,2280){\makebox(0,0)[r]{\strut{}$1$}}%
  \put(750,1850){\makebox(0,0){\strut{}\large\color{blue}$\widehat N_1$}}%
  \put(3650,1950){\makebox(0,0){\strut{}\large$\widehat N_4$}}%
   \put(0,0){\includegraphics[scale=1]{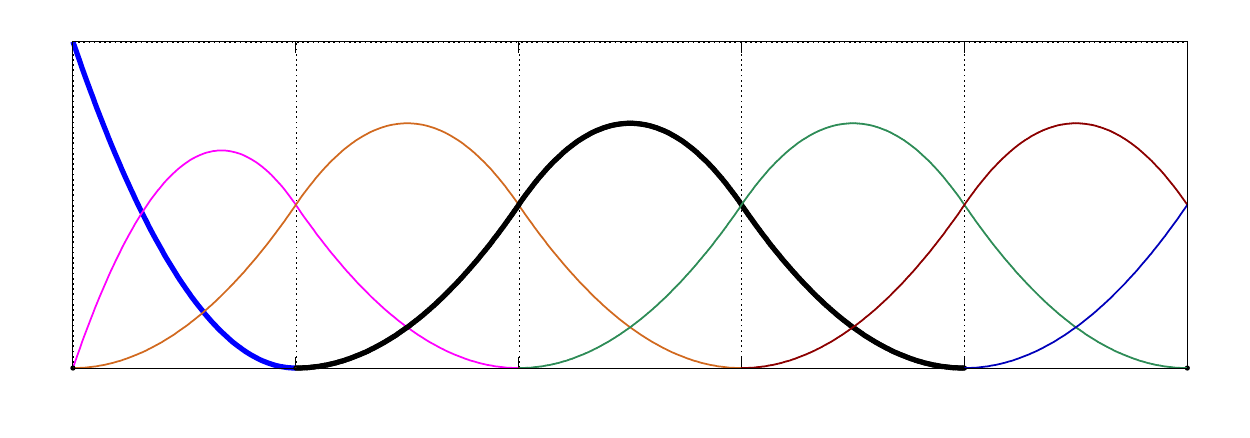}}
\end{picture}%
}
\hbox{
\setlength{\unitlength}{0.0500bp}%
\begin{picture}(7200,2520)(0,0)%
  \put(6839,200){\makebox(0,0){\strut{}$5h$}}%
  \put(5555,200){\makebox(0,0){\strut{}$4h$}}%
  \put(4271,200){\makebox(0,0){\strut{}$3h$}}%
  \put(2988,200){\makebox(0,0){\strut{}$2h$}}%
  \put(1704,200){\makebox(0,0){\strut{}$h$}}%
  \put(420,200){\makebox(0,0){\strut{}$0$}}%
  \put(300,2280){\makebox(0,0)[r]{\strut{}$5$}}%
  \put(300,1425){\makebox(0,0)[r]{\strut{}$0$}}%
  \put(300,570){\makebox(0,0)[r]{\strut{}$-5$}}%
  \put(685,1950){\makebox(0,0){\strut{}\color{blue}\large$\psi_1^{\scriptscriptstyle \rm{ naiv } }$}}%
  \put(3995,2120){\makebox(0,0){\strut{}\large $\psi_4^{\scriptscriptstyle \rm{ naiv } }$}}%
   \put(0,0){\includegraphics[scale=1]{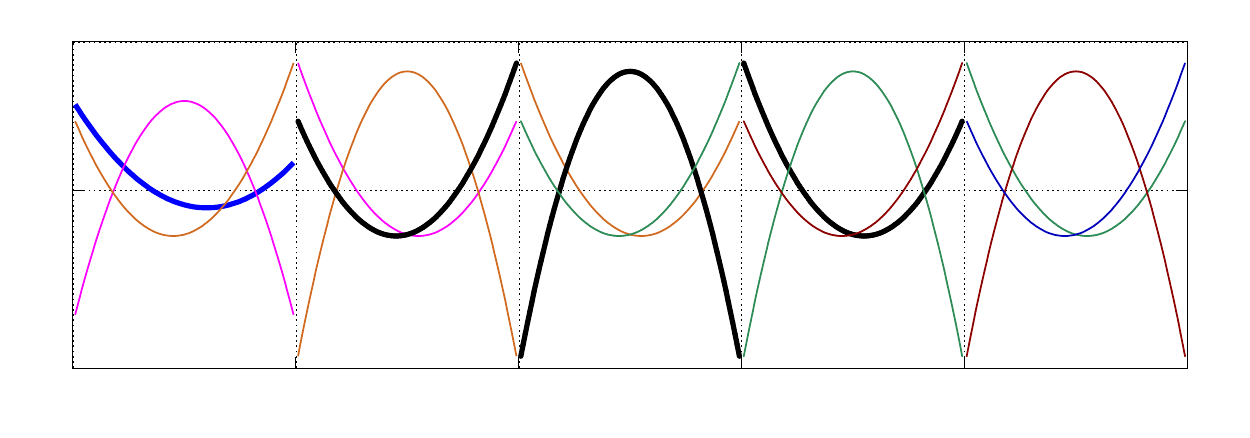}}
\end{picture}%
}
\hbox{
\setlength{\unitlength}{0.0500bp}%
\begin{picture}(7200,2520)(0,0)%
  \put(6839,200){\makebox(0,0){\strut{}$5h$}}%
  \put(5555,200){\makebox(0,0){\strut{}$4h$}}%
  \put(4271,200){\makebox(0,0){\strut{}$3h$}}%
  \put(2988,200){\makebox(0,0){\strut{}$2h$}}%
  \put(1704,200){\makebox(0,0){\strut{}$h$}}%
  \put(420,200){\makebox(0,0){\strut{}$0$}}%
  \put(300,2280){\makebox(0,0)[r]{\strut{}$5$}}%
  \put(300,1340){\makebox(0,0)[r]{\strut{}$0$}}%
  \put(300,400){\makebox(0,0)[r]{\strut{}$-5$}}%
  \put(650,1900){\makebox(0,0){\strut{}\color{blue}\large$\psi_1$}}%
  \put(3650,2060){\makebox(0,0){\strut{}\large $\psi_4$}}%
   \put(0,0){\includegraphics[scale=1]{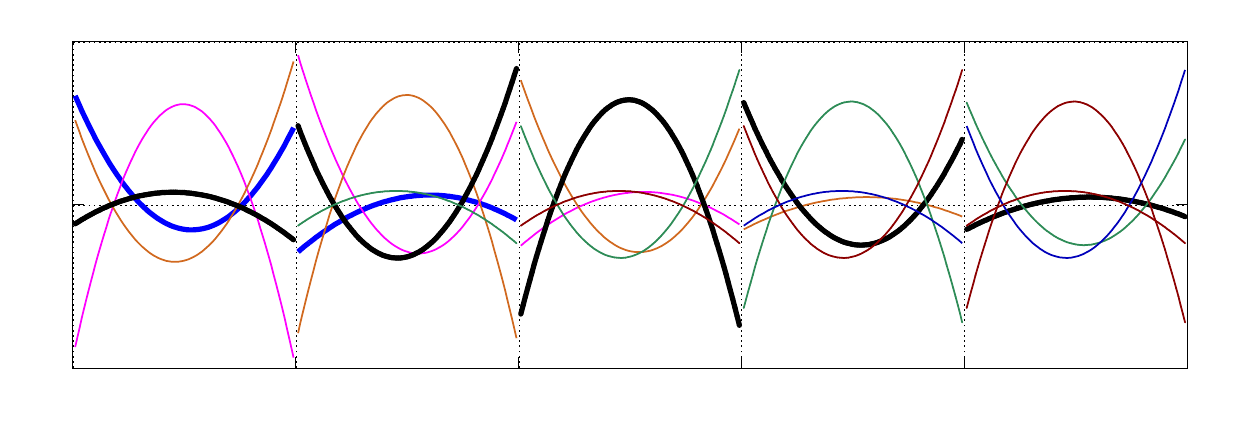}}
\end{picture}%
}
\caption{Primal (top), naive dual (middle) and optimal dual (bottom) basis functions shown on the first five elements of a quadratic spline patch with the first and fourth basis function being highlighted in bold.}
\label{fig:basis_functions}
\end{figure}

%!TEX root = ../article.tex
\subsection{Multilateral construction by tensorization} 
We set $\rho\circ \mathbf F = \NURBSWeight / \det{\nabla_\gammahat \mathbf{F}}$, such that
 $\widehat \rho = 1$. Since the interface in the parametric space is a direct tensor-product of one-dimensional spline spaces, we can construct the biorthogonal basis as a tensor product.

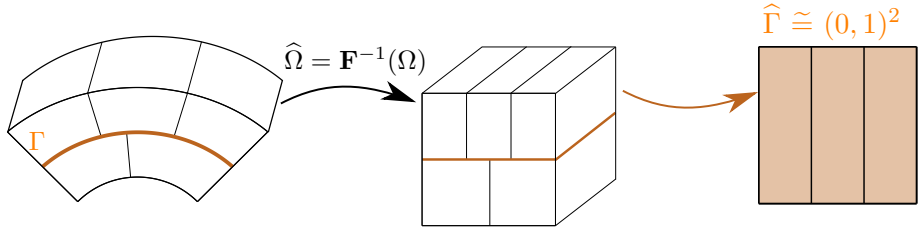
\begin{figure}
\begin{center}
\unitlength 0.60cm
\begin{picture}(20.0,5.0)
\put(0.0,0.5){\includegraphics[height=4\unitlength]{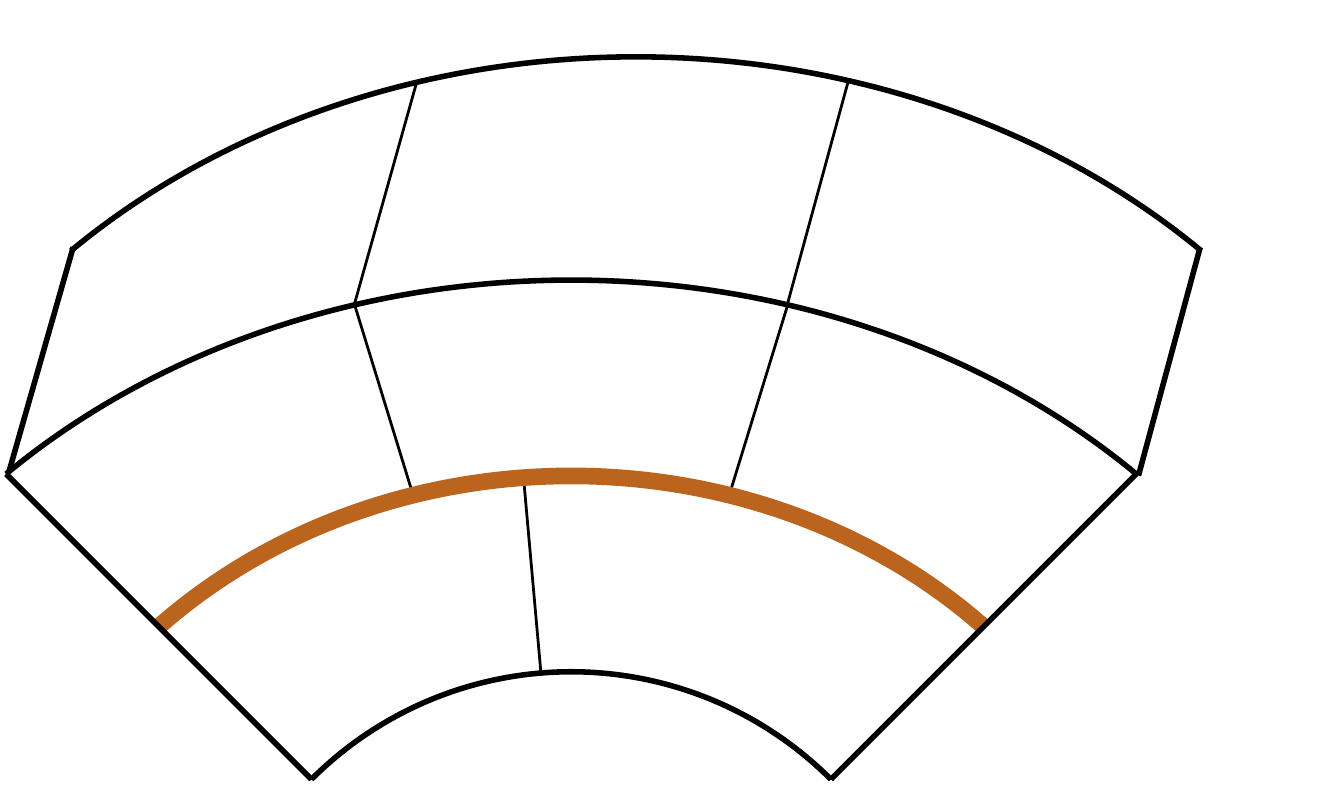}}
\put(9.1,0.0){\includegraphics[height=4\unitlength]{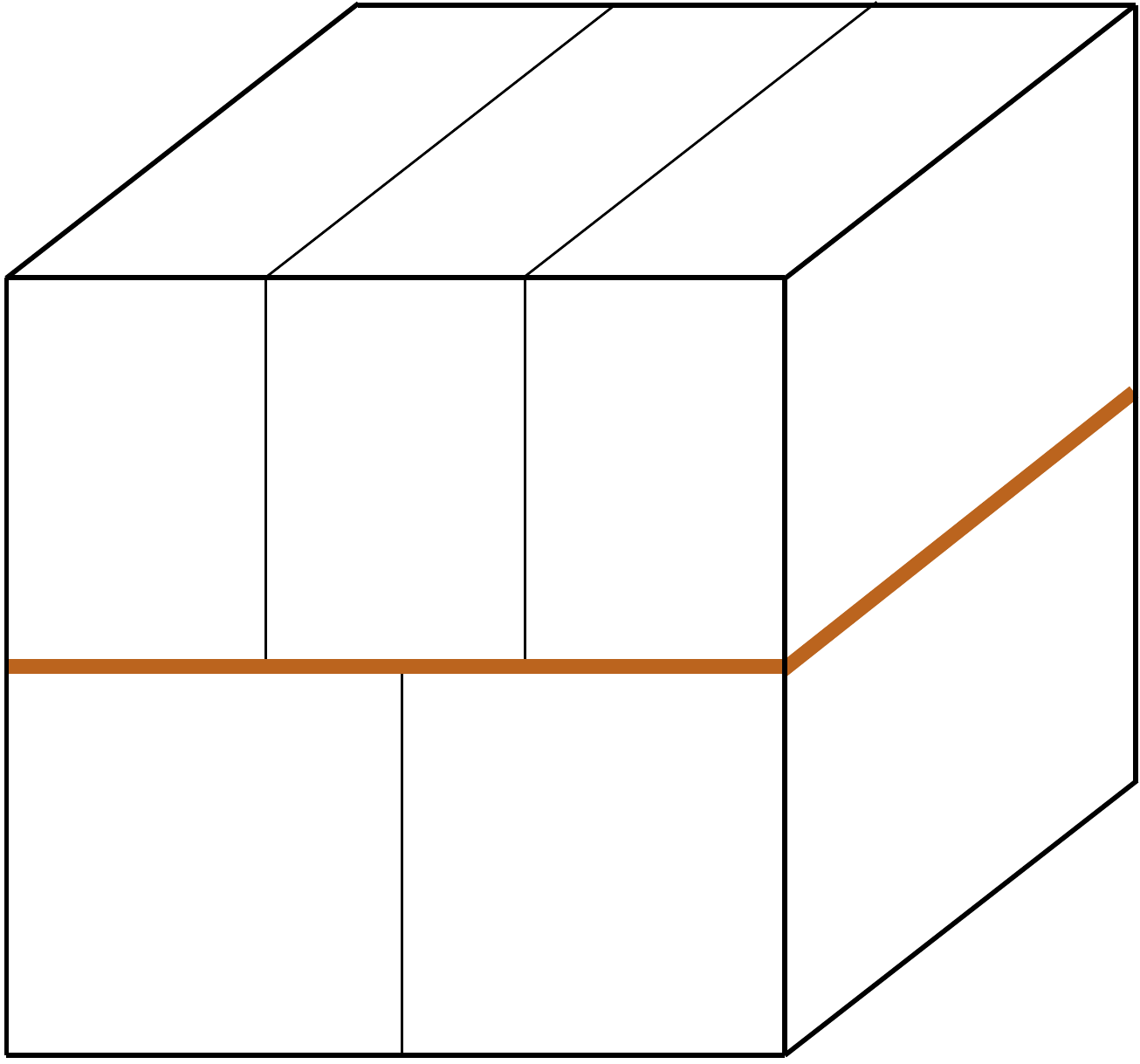}}
\put(16.5,0.5){\includegraphics[height=3.5\unitlength]{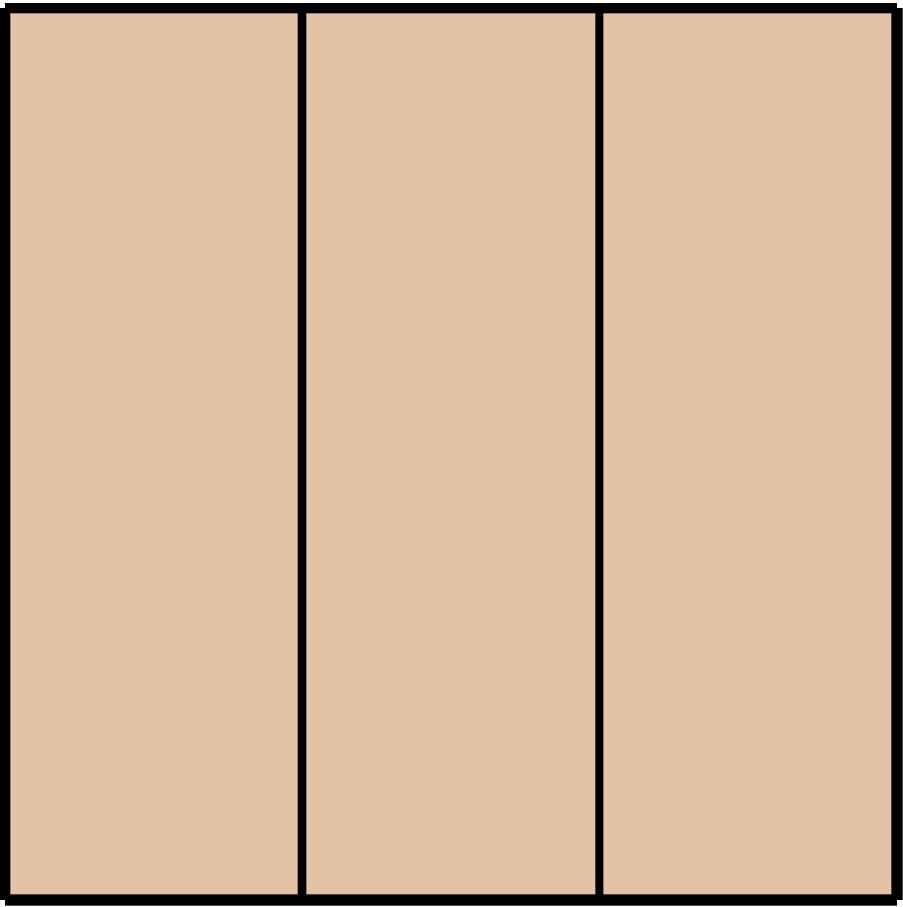}}
\put(0.5,1.7){{$\color{orange}\refInterface$}}
\put(16.6,4.3){\large{$\color{orange}\gammahat ~\widetilde{=}~ (0,1)^2$}}
\put(6.1,3.5){{$\widehat \Omega = \mathbf F^{-1}( \Omega)$}}
\put(6.0,2.7){\includegraphics[width=3\unitlength]{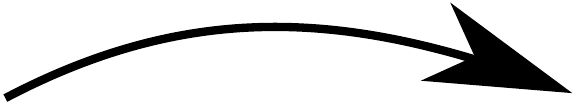}}
\put(13.5,2.5){\includegraphics[width=3\unitlength]{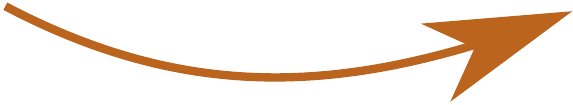}}
\end{picture}
\end{center}
\caption{Illustration of the mapping to the parametric space. }
\label{fig:parametric_interface}
\end{figure} 
 
As a geometric interpretation, $\widehat \rho = 1$ means that the coupling condition is posed in the parametric space, instead of the geometric space, see Figure~\ref{fig:parametric_interface}. This is valid, since the surface measure on the parametric space and the geometric space are mathematically equivalent.
The advantage is that we can directly profit from the tensor product construction on $\refInterface=(0,1)^2\times\{0\}$. With the tensor product B-spline basis
\[
\widehat{B}_{(i,j)}(\zeta_1, \zeta_2) = \widehat{B}_i^p(\zeta_1)\, \widehat{B}_j^p(\zeta_2),
\]
the tensor product of a univariate biorthogonal basis $\psi_i$, viz.
\[
\widehat{\psi}_{(i,j)}(\zeta_1, \zeta_2) = \widehat{\psi}_i^p(\zeta_1)\, \widehat{\psi}_j^p(\zeta_2),
\]
forms a multivariate biorthogonal basis:
\begin{align} \label{eq:2d_biorthogonal}
&\int_\refInterface \widehat{B}_{(i,j)}(\zeta_1, \zeta_2) \, \widehat{\psi}_{(k,l)}(\zeta_1, \zeta_2)~\mathrm{d}(\zeta_1, \zeta_2) \notag \\&\qquad=
\int_0^1 \widehat{B}_{i}(\zeta_1) \widehat{\psi}_{k}(\zeta_1)~\mathrm{d}\zeta_1 \,
\int_0^1 \widehat{B}_{j}(\zeta_2) \widehat{\psi}_{l}(\zeta_2)~\mathrm{d}\zeta_2  
= \delta_{ik} \delta_{jl}.
\end{align}
Of course, the polynomial reproduction order is retained with the tensor product construction.

For different choices of $\rho$, the integrals in~\eqref{eq:2d_biorthogonal} are weighted with $\widehat \rho \neq 1$. Then, in general, the integral cannot be separated into two independent integrals, hence the constructed basis is not biorthogonal.

With two-dimensional interfaces, `crosspoints' are entire boundary faces, due to our regularity assumptions of Section~\ref{subsec:computational_domain}.
By a simple crosspoint modification for the one-dimensional bases that are used in the tensor-product construction, we conveniently get a crosspoint modification also of the two-dimensional basis.
Note that, when the `crosspoints' are only a subset of the boundary faces, still a crosspoint modification can safely be performed on the entire boundary face.
\section{Numerical results}
\label{sec:numerical_results}
%!TEX root = ../article.tex
In the following, we test our newly constructed biorthogonal basis on three numerical examples and compare it with the naive biorthogonal basis from \cite{seitz:16} as well as standard Lagrange multipliers. Where an exact solution is  available, the $L^2$-error is considered, in the other cases convergence is studied	 by observing the internal energy as well as point evaluations. 
All numerical computations are performed with the in-house research code BACI~\cite{baci}.
\subsection{Plate with a hole}
As the first example, we consider the well-known benchmark of an infinite plate with a hole, e.g.~\cite[\S58]{muskhelishvili:77}. Due to symmetry, only a quarter of the plate is considered, and the infinite geometry is cut with the exact traction being applied as a boundary condition. The exact setting is illustrated in Figure~\ref{fig:plateSetup}.

We consider the equations of linear elasticity, where the Cauchy stress $\stress(\mathbf u)$  depends linearly on the strain $\strain(\mathbf u) =\sym{ \nabla \mathbf{u}}$ via Hooke's law as $\stress(\mathbf u) =2\lameM \strain(\mathbf u) + \lameL \mattrace \strain(\mathbf u) \idmat$, with the trace operator $\mattrace \strain = \sum_i \strain_{ii}$ and the Lam\'e parameters $\lameL, \lameM$. 
The Lam\'e parameters can be computed by
\[
\lambda= \frac{\nu E}{(1+\nu)(1-2\nu)},\qquad \mu = \frac{E}{2(1+\nu)}.
\]
We measure convergence in the energy norm
\[
\left\| \mathbf{u} - \mathbf{u}_h \right\|_{\rm E}^2  = \sum_k \int_{\Omega_k} \stress(\mathbf{u} - \mathbf{u}_h ) : \strain(\mathbf{u} - \mathbf{u}_h ) \dx,
\]
where the optimal convergence order is $\mathcal{O}(h^{p})$, see~\cite[Theorem 6]{brivadis:15}.

\begin{figure}[h!]
\begin{center}
\includegraphics[width=.5\textwidth]{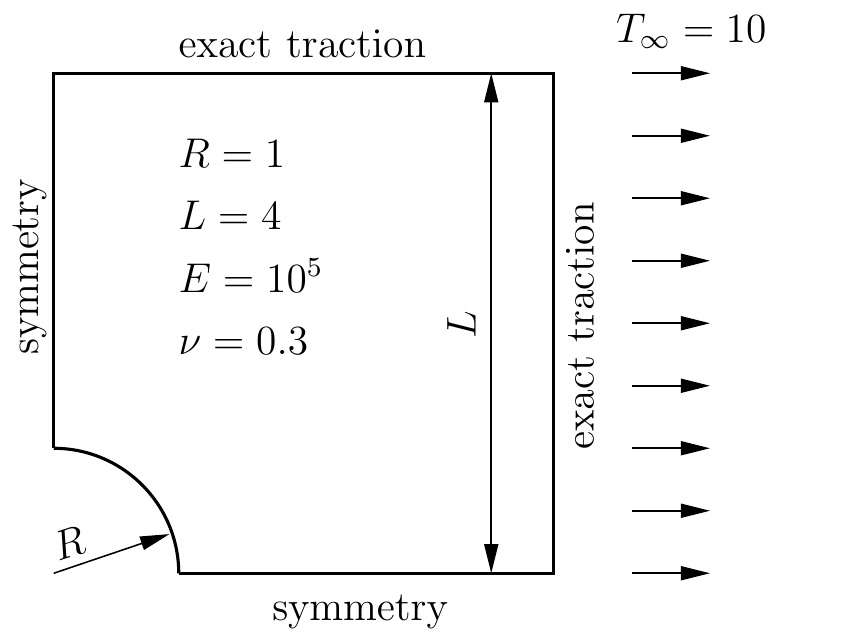}
\caption{Plate with a hole: Geometry, boundary conditions and material parameters.}
\label{fig:plateSetup}
\end{center}
\end{figure}

Different geometry parametrizations are considered, as shown in Figure~\ref{fig:plateGeometries}. 
The first case (Figure~\ref{plateGeometriesA})  is a two-patch setting with a straight interface, where the parametrization of the interface is the same on both subdomains. 
In the second case (Figure~\ref{plateGeometriesB}) the same subdomains are considered, but with a different parametrization, such that the parametrizations along the interface do no longer match. 
We note, that this is a situation, where the construction of~\cite{zou:17} is not exact, but requires additional steps of refinement.
In the third case (Figure~\ref{plateGeometriesC}), the subdomains are coupled across a curved interface.

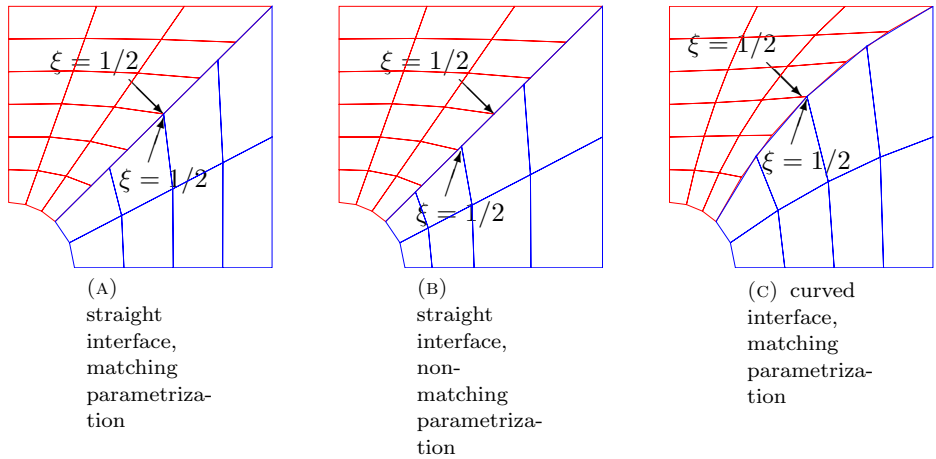
\begin{figure}[h!]
\subfloat[straight interface, matching parametrization]
{
\hbox{
\setlength{\unitlength}{0.04300bp}%
 \begin{picture}(2300,2200)(0,0)%
  \put(750,1800){\makebox(0,0){$\xi=1/2$}}%oben
  \put(1050,1650){\vector(1,-1){300}}%oben
  \put(1350,770){\makebox(0,0){$\xi=1/2$}}% unten
  \put(1220,930){\vector(1,3){130}}%unten
   \put(0,0){\includegraphics[width = 2300\unitlength]{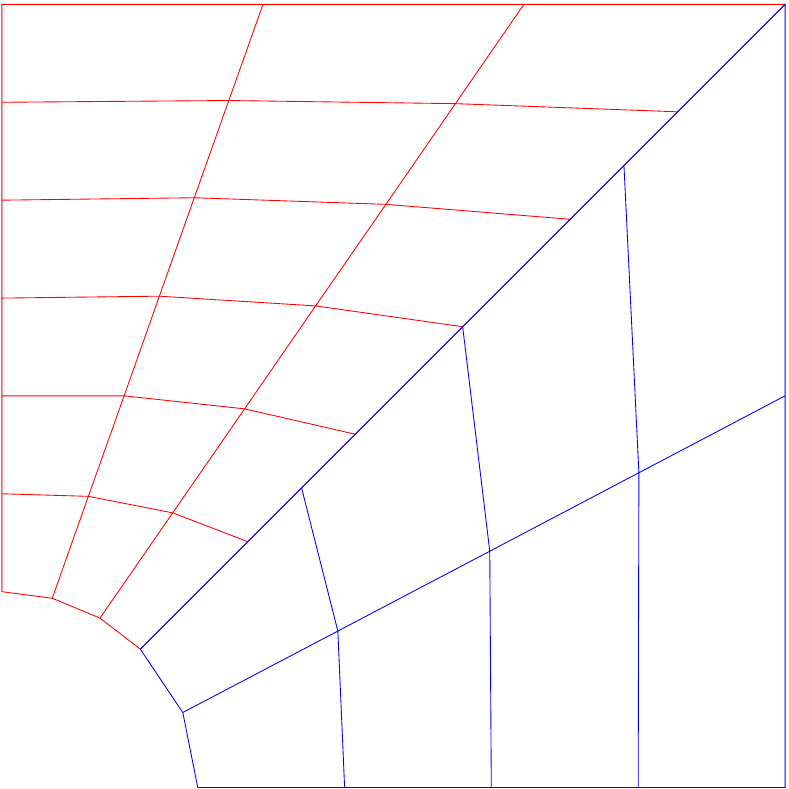}}
\end{picture}
}
\label{plateGeometriesA}
}\hfill
\subfloat[straight interface, non-matching parametrization]
{
\hbox{
\setlength{\unitlength}{0.04300bp}%
 \begin{picture}(2300,2200)(0,0)%
  \put(750,1800){\makebox(0,0){$\xi=1/2$}}%oben
  \put(1050,1650){\vector(1,-1){300}}%oben
  \put(1060,470){\makebox(0,0){$\xi=1/2$}}% unten
  \put(930,630){\vector(1,3){130}}%unten
   \put(0,0){\includegraphics[width = 2300\unitlength]{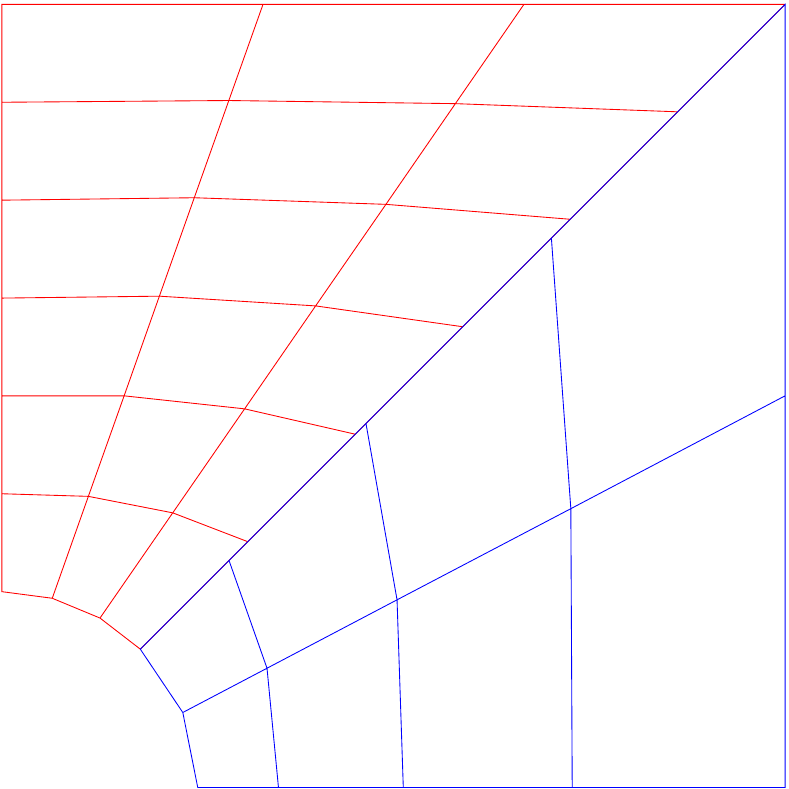}}
\end{picture}
}
\label{plateGeometriesB}
}\hfill
\subfloat[curved interface, matching parametrization]
{
\hbox{
\setlength{\unitlength}{0.04300bp}%
 \begin{picture}(2300,2200)(0,0)%
  \put(550,1950){\makebox(0,0){$\xi=1/2$}}%oben
  \put(850,1800){\vector(1,-1){300}}%oben
  \put(1195,920){\makebox(0,0){$\xi=1/2$}}% unten
  \put(1065,1080){\vector(1,3){130}}%unten
   \put(0,0){\includegraphics[width = 2300\unitlength]{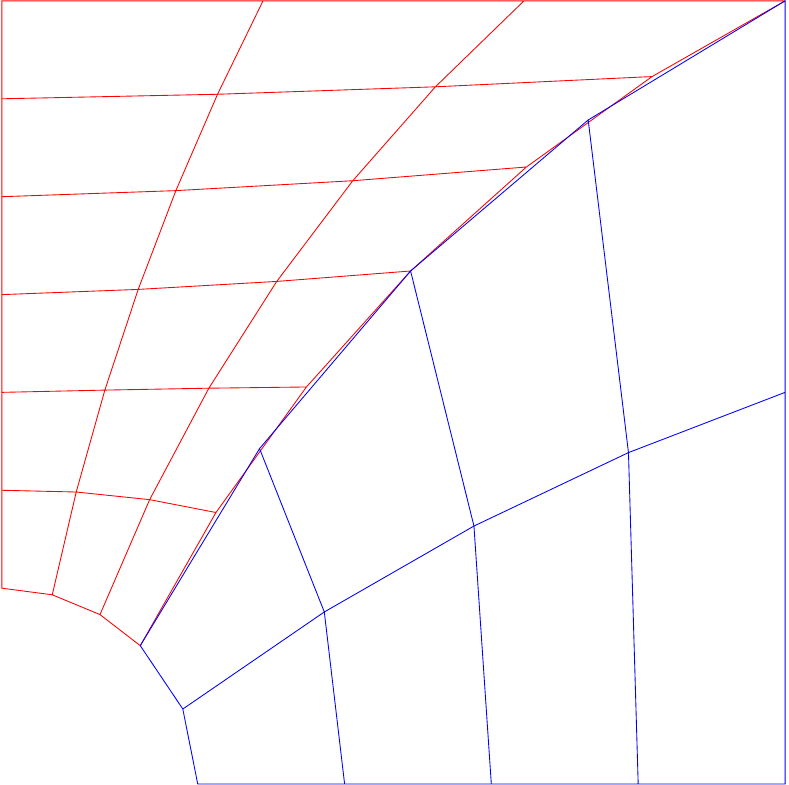}}
\end{picture}
}
\label{plateGeometriesC}
}
\caption{Different geometric setups of the plate with hole, shown for the mesh ratio $2:3$. }
\label{fig:plateGeometries}
\end{figure}

For all three setups, convergence for quadratic splines and different Lagrange multiplier bases is presented in Figure~\ref{fig:convergence_p_2}.
We observe an optimal order convergence for the newly constructed biorthogonal basis functions ('optimal'), with similar error values as with a standard Lagrange multiplier ('std'). In comparison, the naive, element-wise biorthogonal basis functions ('ele dual') as considered in~\cite{seitz:16} show suboptimal convergence, especially when  the slave mesh is coarser than the master mesh.

\begin{figure}
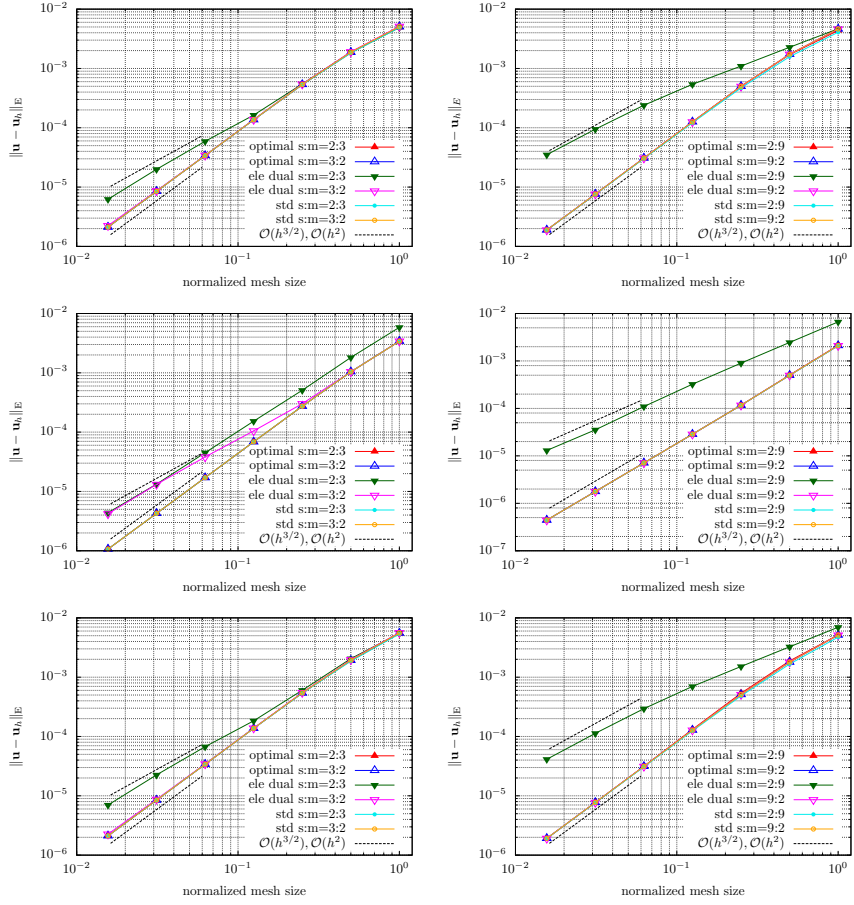

\scalebox{.5}{
\input{figures/plate/straight_MatchingParam_2nd/convergence.tex}
\input{figures/plate/straight_MatchingParam_2nd_2.9/convergence.tex}
}
\\
\scalebox{.5}{
\input{figures/plate/straight_nonMatchingParam_2nd/convergence.tex}
\input{figures/plate/straight_nonMatchingParam_2nd_2.9/convergence.tex}
}\\
\scalebox{.5}{
\input{figures/plate/curved_MatchingParam_2nd/convergence.tex}
\input{figures/plate/curved_MatchingParam_2nd_2.9/convergence.tex}
}
\caption{Convergence of the energy error for $p=2$. Comparison of the new biorthogonal basis ('optimal'), standard dual basis ('ele dual') and standard Lagrange multipliers ('std'). Left: mesh ratio $2:3$. Right: mesh ratio $2:9$. From top to bottom the three mesh cases (a) to (c) from Figure \ref{fig:plateGeometries} are considered. }
\label{fig:convergence_p_2}
\end{figure}

The same comparisons for cubic splines are  shown in Figure~\ref{fig:convergence_p_3}. Again, we see an optimal order convergence of the optimal biorthogonal basis functions as expected theoretically, while the suboptimality of the naive biorthogonal basis functions from~\cite{seitz:16} becomes even more apparent.
However, when the slave mesh is coarser than the master mesh, the values of the error are larger than for the standard Lagrange multiplier case. Since the gap gets smaller with further refinements, this seems to be a pre-asymptotical effect. We note that for a fine slave mesh, no significant suboptimality can be observed in the pre-asymptotics.

\begin{figure}
\scalebox{.5}{
\input{figures/plate/straight_MatchingParam_3rd/convergence.tex}
\input{figures/plate/straight_MatchingParam_3rd_2.9/convergence.tex}
}
\\
\scalebox{.5}{
\input{figures/plate/straight_nonMatchingParam_3rd/convergence.tex}
\input{figures/plate/straight_nonMatchingParam_3rd_2.9/convergence.tex}
}\\
\scalebox{.5}{
\input{figures/plate/curved_MatchingParam_3rd/convergence.tex}
\input{figures/plate/curved_MatchingParam_3rd_2.9/convergence.tex}
}
\caption{Convergence of the energy error for $p=3$. Comparison of the new biorthogonal basis ('optimal'), standard dual basis ('ele dual') and standard Lagrange multipliers ('std'). Left: mesh ratio $2:3$. Right: mesh ratio $2:9$. From top to bottom the three mesh cases (a) to (c) from Figure \ref{fig:plateGeometries} are considered. }
\label{fig:convergence_p_3}
\end{figure}

In summary,  we have observed optimal convergence rates in all cases for the newly constructed biorthogonal basis functions, while suboptimal rates were seen for the naive biorthogonal basis from \cite{seitz:16}. When the finer side is chosen as the slave side, the error values were the same as for standard Lagrange multipliers. Only when the slave side is coarser, a suboptimal preasymptotic evolves. Hence, for the optimal biorthogonal basis, it is especially important to choose the finer side of the interface as the slave side, whenever this is possible. 

\subsection{Bimaterial annulus} \label{sec:numerics:annulus}
In the second example, we consider a two-dimensional bimaterial setting. The bimaterial annulus shown in  Figure~\ref{fig:geometry_bimaterial} consists of a soft material ($E_1=	1\mathrm{e}3$, $\nu_1 = 0.3$) with a thin hard  inclusion ($E_2=1\mathrm{e}5$, $\nu_2=0.3$) with an elliptic interface.  
The considered geometry parameters are:
\begin{align*}
r_{\rm i}=0.75, \quad
&r_{\rm o}=1,
\\
a_1  = 0.55\,(r_{\rm i}+r_{\rm o}-1) = 0,95975 ,\quad
&b_1 =  0.5\,(r_{\rm i}+r_{\rm o}-1)/1.1 \approx 0,7932,\\
a_2 = 0.55\,(r_{\rm i}+r_{\rm o}+1) = 0,96525, \quad
&b_2 = 0.5\,(r_{\rm i}+r_{\rm o}+1)/1.1 \approx 0,7977.
\end{align*}
Inside the annulus, a constant unit-pressure is applied, and  the outer boundary is a homogeneous Neumann boundary. The rigid body modes are removed by restricting the corresponding deformations.

The different stiffnesses and the thin geometry of the inclusion demand for anisotropic elements and  different mesh-sizes  in the different subdomains. The interior subdomain consists of $20$ elements in the  angular direction and three elements in the radial direction, the thin inclusion consists of $68$ elements in the angular direction and one element in the radial direction, and the outer subdomain consists of $24$ element in the angular direction and two elements in the radial direction.

\begin{figure}
\begin{center}
\scalebox{.95}{%!TEX root = ../../article.tex
\begin{tikzpicture} 
\setlength{\unitlength}{3cm}%
% Kreise
\draw (0,0) circle [radius=1.0\unitlength];
\draw (0,0) circle [radius=0.75\unitlength];
% Ellipsen
\draw (0,0) circle [x radius=0.95975\unitlength, y radius=0.7932\unitlength] ;
\draw (0,0) circle [x radius=0.96525\unitlength, y radius=0.7977\unitlength];
\end{tikzpicture}}\hspace{1em}
\scalebox{.95}{%!TEX root = ../../article.tex

\begin{tikzpicture} 
\setlength{\unitlength}{1.0cm}%
\newlength{\neumannsize}
\setlength{\neumannsize}{0.85\unitlength}
% Kreise
\draw (0,0) circle [radius=3.0\unitlength];
\draw (0,0) circle [radius=1.0\unitlength];
% Ellipsen
\draw (0,0) circle [x radius=2.0\unitlength, y radius=1.4\unitlength] ;
\draw (0,0) circle [x radius=2.5\unitlength, y radius=2.2\unitlength];
%x-achsen Ellipsen 
\draw[->] (0.1\unitlength,0) -- (1.9\unitlength, 0) node[pos=0.7,above]{$a_1$};
\draw[->] (-0.1\unitlength,0) -- (-2.4\unitlength, 0) node[pos=0.5,above]{$a_2$};
% y-achsen ellipsen
\draw[->] (0,0.1\unitlength) -- (0,1.3\unitlength) node[pos=0.90,right]{$b_1$};
\draw[->] (0,-0.1\unitlength) -- (0,-2.1\unitlength) node[pos=0.54,right]{$b_2$};
% Kreise
\draw[->] (.07\unitlength,0.07\unitlength) -- (2.05\unitlength,2.05\unitlength) node[pos=0.55,above]{$r_{\rm o}~$};
\draw[->] (-0.07\unitlength, .07\unitlength) -- (-0.64\unitlength,0.64\unitlength) node[pos=0.4,above]{$r_{\rm i}$};
%Inner Pressure
%%% Erster Quadrant
\draw[->, red] (0.9898\neumannsize, 0.1423\neumannsize) -- (0.9898\unitlength,0.1423\unitlength);
\draw[->, red] (0.9595\neumannsize, 0.2817\neumannsize) -- (0.9595\unitlength,0.2817\unitlength);
\draw[->, red] (0.9096\neumannsize, 0.4154\neumannsize) -- (0.9096\unitlength,0.4154\unitlength);
\draw[->, red] (0.8413\neumannsize, 0.5406\neumannsize) -- (0.8413\unitlength,0.5406\unitlength);
\draw[->, red] (0.7557\neumannsize, 0.6549\neumannsize) -- (0.7557\unitlength,0.6549\unitlength);
\draw[->, red] (0.6549\neumannsize, 0.7557\neumannsize) -- (0.6549\unitlength,0.7557\unitlength);
\draw[->, red] (0.5406\neumannsize, 0.8413\neumannsize) -- (0.5406\unitlength,0.8413\unitlength);
\draw[->, red] (0.4154\neumannsize, 0.9096\neumannsize) -- (0.4154\unitlength,0.9096\unitlength);
\draw[->, red] (0.2817\neumannsize, 0.9595\neumannsize) -- (0.2817\unitlength,0.9595\unitlength);
\draw[->, red] (0.1423\neumannsize, 0.9898\neumannsize) -- (0.1423\unitlength,0.9898\unitlength);

%%% Zweiter Quadrant
\draw[->, red] (-0.9898\neumannsize, 0.1423\neumannsize) -- (-0.9898\unitlength,0.1423\unitlength);
\draw[->, red] (-0.9595\neumannsize, 0.2817\neumannsize) -- (-0.9595\unitlength,0.2817\unitlength);
\draw[->, red] (-0.9096\neumannsize, 0.4154\neumannsize) -- (-0.9096\unitlength,0.4154\unitlength);
\draw[->, red] (-0.8413\neumannsize, 0.5406\neumannsize) -- (-0.8413\unitlength,0.5406\unitlength);
\draw[->, red] (-0.7557\neumannsize, 0.6549\neumannsize) -- (-0.7557\unitlength,0.6549\unitlength);
\draw[->, red] (-0.6549\neumannsize, 0.7557\neumannsize) -- (-0.6549\unitlength,0.7557\unitlength);
\draw[->, red] (-0.5406\neumannsize, 0.8413\neumannsize) -- (-0.5406\unitlength,0.8413\unitlength);
\draw[->, red] (-0.4154\neumannsize, 0.9096\neumannsize) -- (-0.4154\unitlength,0.9096\unitlength);
\draw[->, red] (-0.2817\neumannsize, 0.9595\neumannsize) -- (-0.2817\unitlength,0.9595\unitlength);
\draw[->, red] (-0.1423\neumannsize, 0.9898\neumannsize) -- (-0.1423\unitlength,0.9898\unitlength);

%%% Dritter Quadrant
\draw[->, red] (-0.9898\neumannsize, -0.1423\neumannsize) -- (-0.9898\unitlength,-0.1423\unitlength);
\draw[->, red] (-0.9595\neumannsize, -0.2817\neumannsize) -- (-0.9595\unitlength,-0.2817\unitlength);
\draw[->, red] (-0.9096\neumannsize, -0.4154\neumannsize) -- (-0.9096\unitlength,-0.4154\unitlength);
\draw[->, red] (-0.8413\neumannsize, -0.5406\neumannsize) -- (-0.8413\unitlength,-0.5406\unitlength);
\draw[->, red] (-0.7557\neumannsize, -0.6549\neumannsize) -- (-0.7557\unitlength,-0.6549\unitlength);
\draw[->, red] (-0.6549\neumannsize, -0.7557\neumannsize) -- (-0.6549\unitlength,-0.7557\unitlength);
\draw[->, red] (-0.5406\neumannsize, -0.8413\neumannsize) -- (-0.5406\unitlength,-0.8413\unitlength);
\draw[->, red] (-0.4154\neumannsize, -0.9096\neumannsize) -- (-0.4154\unitlength,-0.9096\unitlength);
\draw[->, red] (-0.2817\neumannsize, -0.9595\neumannsize) -- (-0.2817\unitlength,-0.9595\unitlength);
\draw[->, red] (-0.1423\neumannsize, -0.9898\neumannsize) -- (-0.1423\unitlength,-0.9898\unitlength);
\draw[->, red] (-0.1423\neumannsize, -0.9898\neumannsize) -- (-0.1423\unitlength,-0.9898\unitlength);

%%% Vierter Quadrant
\draw[->, red] (0.9898\neumannsize, -0.1423\neumannsize) -- (0.9898\unitlength,-0.1423\unitlength);
\draw[->, red] (0.9595\neumannsize, -0.2817\neumannsize) -- (0.9595\unitlength,-0.2817\unitlength);
\draw[->, red] (0.9096\neumannsize, -0.4154\neumannsize) -- (0.9096\unitlength,-0.4154\unitlength);
\draw[->, red] (0.8413\neumannsize, -0.5406\neumannsize) -- (0.8413\unitlength,-0.5406\unitlength);
\draw[->, red] (0.7557\neumannsize, -0.6549\neumannsize) -- (0.7557\unitlength,-0.6549\unitlength);
\draw[->, red] (0.6549\neumannsize, -0.7557\neumannsize) -- (0.6549\unitlength,-0.7557\unitlength);
\draw[->, red] (0.5406\neumannsize, -0.8413\neumannsize) -- (0.5406\unitlength,-0.8413\unitlength);
\draw[->, red] (0.4154\neumannsize, -0.9096\neumannsize) -- (0.4154\unitlength,-0.9096\unitlength);
\draw[->, red] (0.2817\neumannsize, -0.9595\neumannsize) -- (0.2817\unitlength,-0.9595\unitlength);
\draw[->, red] (0.1423\neumannsize, -0.9898\neumannsize) -- (0.1423\unitlength,-0.9898\unitlength);

%axis
\draw[->, red] (0, \neumannsize) -- (0,\unitlength);
\draw[->, red] (0, -\neumannsize) -- (0,-\unitlength);
\draw[->, red] (\neumannsize,0) -- (\unitlength,0);
\draw[->, red] (-\neumannsize,0) -- (-\unitlength,0);

\draw (.41\unitlength,-.22\unitlength) node{\color{red}\scriptsize $\widehat b_0 = 1$};

\draw (-1.3\unitlength,0.6\unitlength) node{$E_1, \nu_1$};
\draw (-1.3\unitlength,1.4\unitlength) node{$E_2, \nu_2$};
\draw (-1.3\unitlength,2.2\unitlength) node{$E_1, \nu_1$};

\end{tikzpicture}} 
\caption{Geometry and problem setting. Left: Exact geometry. Right: Schematic plot with applied boundary traction. }
\label{fig:geometry_bimaterial}
\end{center}
\end{figure}
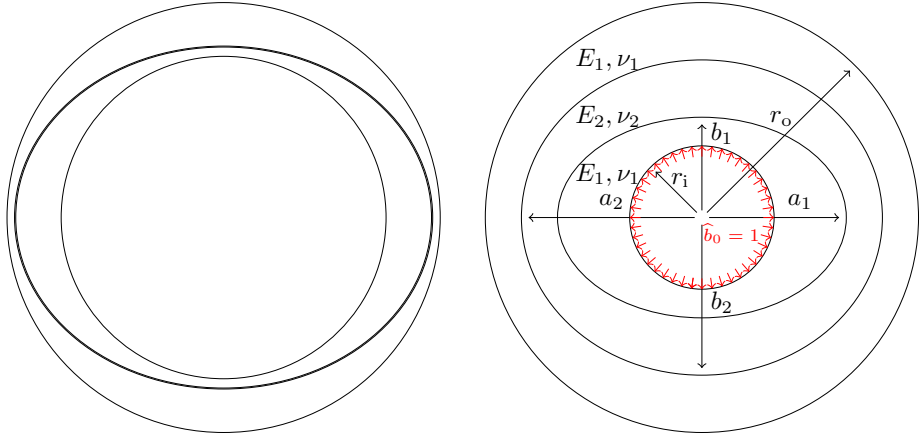

In Figure~\ref{fig:convergence_annulus}, convergence of the energy $E = \int_\Omega \stress(\mathbf u):\strain(\mathbf u) \dx$ is presented. Lacking the exact solution, we use a reference value on a refined mesh $E_{\rm ref}\approx 3.59\,\mathrm{e}{-3}$.
Again, we clearly see the suboptimality of the naive biorthogonal basis functions from \cite{seitz:16}, which exhibit a convergence of the order $h^{3/2}$.
In the second order case, the optimal biorthogonal basis shows the same approximation quality as the standard Lagrange multiplier, independent of the choice of the slave side. In the third order case, a suboptimal pre-asymptotic can again be observed for the case of a coarse slave mesh. Still, the error in the energy is smaller than for the naive biorthogonal basis.  When the fine side is chosen as the  slave space, again no difference in the approximation quality is seen between the optimal dual basis and the standard Lagrange multiplier.

\begin{figure}[h!]
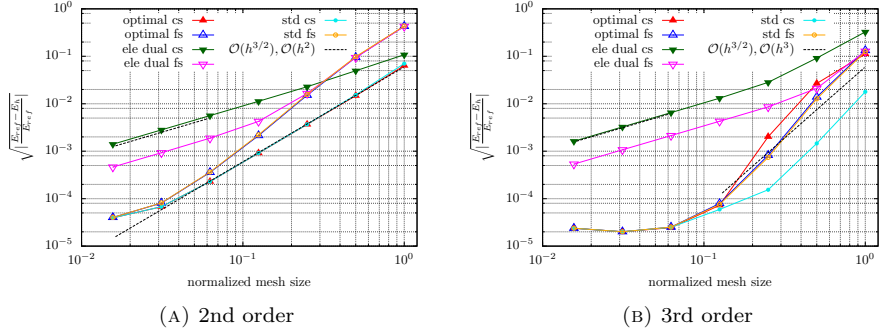

\subfloat[2nd order]
{
\scalebox{.5}{
\input{figures/biMatAn/convergence_energy_2nd.tex}
}
}
\subfloat[3rd order]
{
\scalebox{.5}{
\input{figures/biMatAn/convergence_energy_3rd.tex}
}
}
\caption{Convergence of the energy error for $p=2$ and $p=3$.
In the fine   slave ('fs') case, the thin inclusion layer with a finer (interface-) mesh is chosen as slave  side (on both interfaces).
In the coarse slave ('cs') case, the thin inclusion layer with a finer (interface-) mesh is chosen as master side (on both interfaces).}
\label{fig:convergence_annulus}
\end{figure}

These results are of particular importance when comparing to~\cite{flemisch:04,flemisch:07}, where strong oscillations could be observed for curved interfaces for the finite element case. The reason why this is not observed here for isogeometric methods  might be the exact representation of the curved interface. 

\subsection{Pressurized hollow sphere}
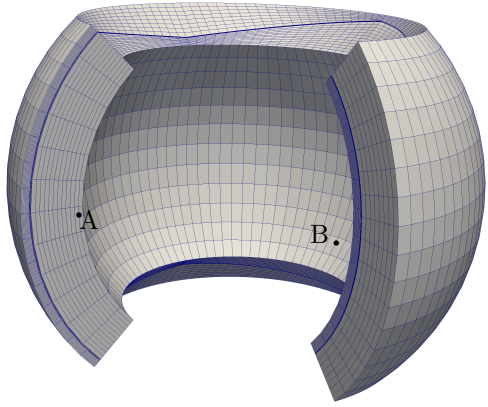
\begin{figure}[htb!]
\begin{center}
\begin{picture}(160,180)
\put(0,0){\includegraphics[width=0.55\textwidth]{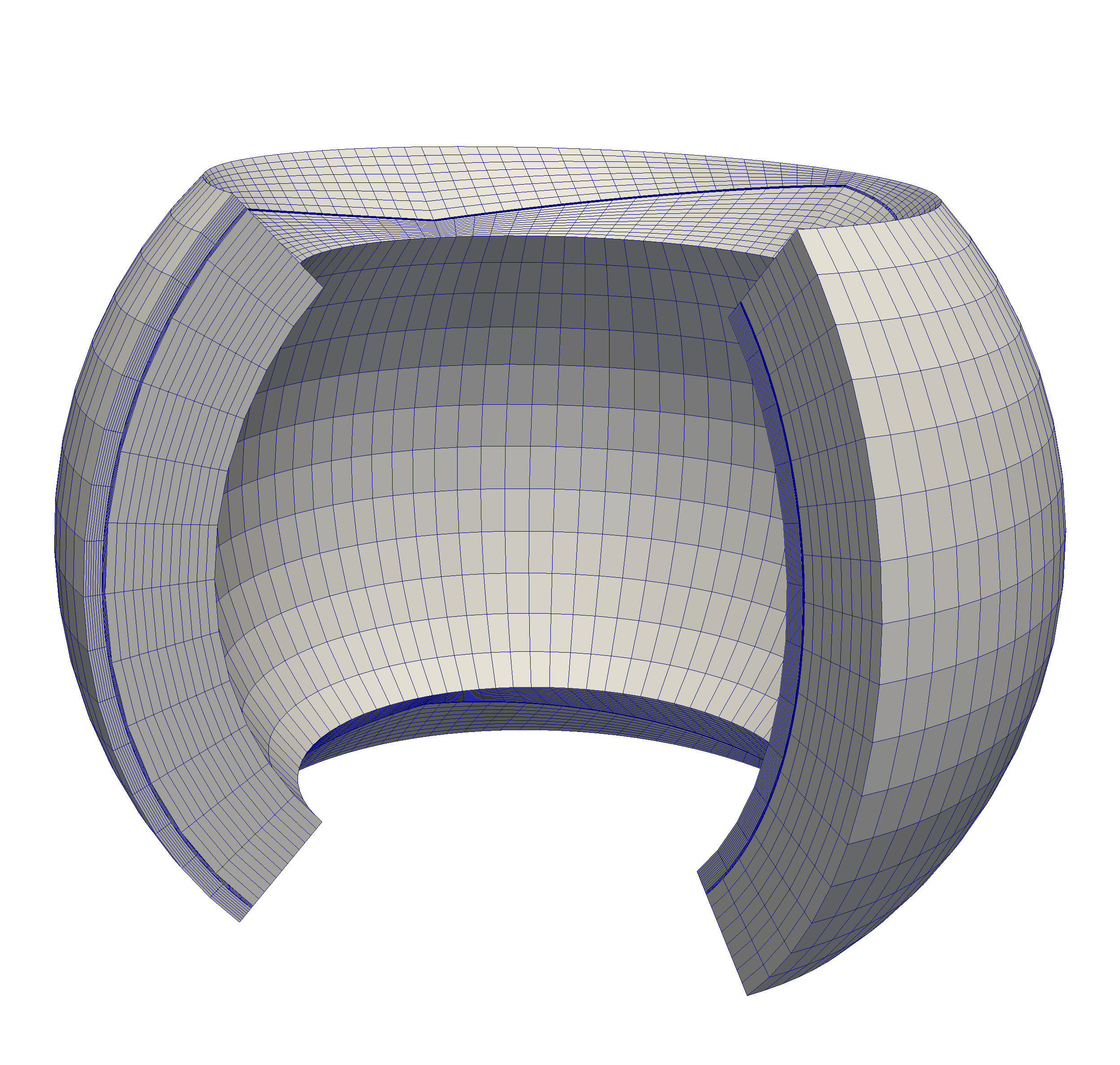}}
\put(36.67,82.19){\circle*{2}}
\put(36.67,77){\text{A}}
\put(132.92,71.6){\circle*{2}}
\put(123.14,71.6){\text{B}}
\end{picture}
\caption{Bimaterial geometry with the front  quarter removed for visualization.}
\label{fig:geo_3d}
\end{center}
\end{figure}
The final example extends the previous setup to the three-dimensional case. We consider a   pressurized hollow sphere with two $45^\circ$ holes as shown in Figure~\ref{fig:geo_3d}. Again, there is a thin inclusion of a stiff material with an elliptic cross-section. More precisely, the equatorial plane resembles the two-dimensional geometry from Section \ref{sec:numerics:annulus}.   We choose the same material parameters as before, but consider a non-linear Neo-Hooke material:  
\[
 \sef(\rcg) = c\, (\mathrm{tr}\rcg -3) + \frac{c}{\beta}\big((\det\rcg)^{-\beta}-1\big),
\]
with $c={E}/(1+4\nu)$ and $\beta={\nu}/(1-2\nu)$.

%The full geometry  is discretized with four vertical lines of $C^0$-continuity and only one quarter of the geometry needs to be computed by symmetry reasons. Rotational modes are again removed by appropriate restrictions of the linear system. 
The final deformation for quadratic NURBS on $68\,624$ elements for the whole domain, with $104\,016$ control points is shown in Figure~\ref{fig:3d_result_stress}, which includes the circumferential Cauchy stress. As expected, the thin stiff inclusion carries most of the pressure. 
The biorthogonal basis guarantees an accurate and smooth transmission of the forces, and no oscillations across the interface can be seen at all. 

\begin{figure}[htb!]
\begin{minipage}{.48\textwidth}
\scalebox{0.53}{
\input{figures/biMat_3D/displacement_plot.tex}
}
\caption{Radial displacement  in relation to the internal pressure. }
\label{fig:3d_load_displacement}
\end{minipage}\hfill
\begin{minipage}{.48\textwidth}
\scalebox{0.53}{
\input{figures/biMat_3D/displacement_abs_error.tex}
}
\caption{Estimated error of the displacements by comparison of $\bar{h}=1$ to $\bar{h}=0.25$.}
\label{fig:3d_absolut_difference}
\end{minipage}
\end{figure}

We observe the radial displacements at two points $A$ and $B$ as shown in Figure~\ref{fig:geo_3d} during the increase of the internal pressure, see Figure~\ref{fig:3d_load_displacement}.
The discretization error is estimated qualitatively by comparing to the values obtained on a coarser mesh, see Figure~\ref{fig:3d_absolut_difference}. We see a good agreement, since the computed displacements differ by a value of less then $0.02$. For pressure values lower than $150$, the difference is even less then 2e-4.

 \begin{figure}
 \centering
 \includegraphics[width=\textwidth]{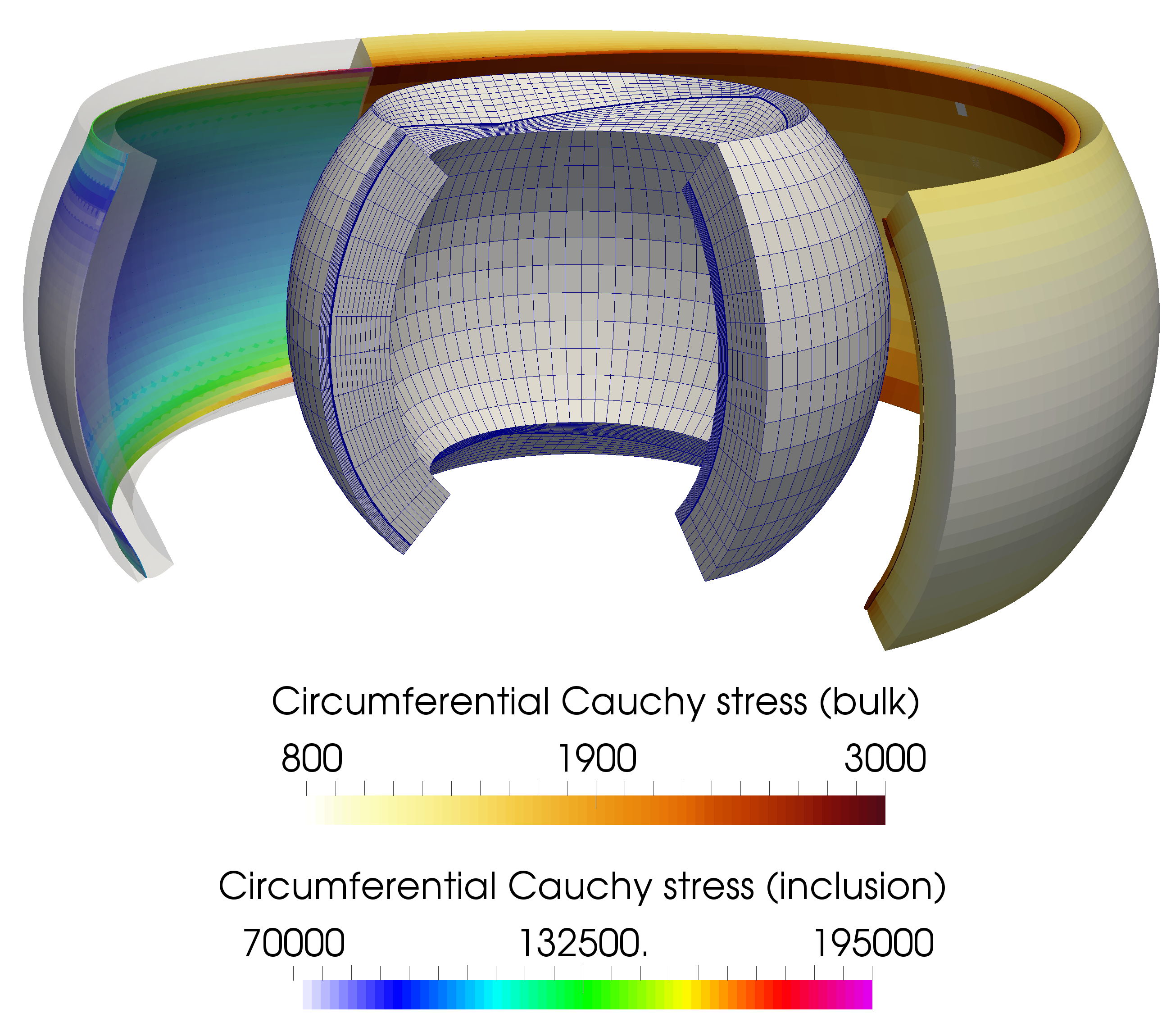}
\caption{Deformed geometry with the circumferential Cauchy stress indicated.}
\label{fig:3d_result_stress}
\end{figure}

\section{Conclusions}
\label{sec:conclusion}
We have investigated isogeometric methods with a newly constructed biorthogonal basis  that yields optimal convergence rates. Thanks to the local support of the dual basis, the resulting equation system is sparse and elliptic. The new biorthogonal basis is proposed with a univariate construction, that is then adapted for two-dimensional interfaces by a tensor product.  To preserve biorthogonality in the tensor product, we have to consider equivalent, weighted integrals. A crosspoint modification is inherently included in the one-dimensional construction.

The numerical results include finite deformations in 3D and confirm the optimal convergence. They also show that, whenever possible, the slave side should be chosen as the finer mesh, since a suboptimal pre-asymptotic is observed for coarse slave spaces. Finally, a three-dimensional bimaterial example with finite deformations qualitatively confirms the suitability and efficiency of the method for large-scale engineering applications.

\bibliography{lit.bib}

\begin{thebibliography}{10}
\expandafter\ifx\csname url\endcsname\relax
  \def\url#1{\texttt{#1}}\fi
\expandafter\ifx\csname urlprefix\endcsname\relax\def\urlprefix{URL }\fi
\expandafter\ifx\csname href\endcsname\relax
  \def\href#1#2{#2} \def\path#1{#1}\fi

\bibitem{cottrell:06}
J.~A. Cottrell, A.~Reali, Y.~Bazilevs, T.~J.~R. Hughes, Isogeometric analysis
  of structural vibrations, Comput. Methods Appl. Mech. Eng. 195~(41--43)
  (2006) 5257 -- 5296.

\bibitem{hughes:09}
J.~A. Cottrell, T.~J.~R. Hughes, Y.~Bazilevs, Isogeometric Analysis. Towards
  Integration of CAD and FEA, Wiley, Chichester, 2009.

\bibitem{beirao:14}
L.~Beir{\~a}o Da~Veiga, A.~Buffa, G.~Sangalli, R.~V\'{a}zquez, Mathematical
  analysis of variational isogeometric methods, Acta Numer. 23 ({2014})
  157--287.

\bibitem{nguyen:15}
V.~P. Nguyen, C.~Anitescu, S.~P.~A. Bordas, T.~Rabczuk, Isogeometric analysis:
  An overview and computer implementation aspects, Math. Comp. Simul. 117
  (2015) 89 -- 116.

\bibitem{Hoellig03}
K.~H\"ollig, {Finite Element Methods with B-Splines}, Frontiers in Applied
  Mathematics, SIAM, 2003.

\bibitem{nguyen:14}
V.~P. Nguyen, P.~Kerfriden, M.~Brino, S.~P.~A. Bordas, E.~Bonisoli, Nitsche's
  method for two and three dimensional {NURBS} patch coupling, Comput. Mech.
  53~(6) (2014) 1163--1182.

\bibitem{bletzinger:14}
A.~Apostolatos, R.~Schmidt, R.~W{\"u}chner, K.-U. Bletzinger, {A Nitsche-type
  formulation and comparison of the most common domain decomposition methods in
  isogeometric analysis}, Int. J. Numer. Methods Eng. 97 (2014) 473--504.

\bibitem{hofer:16}
C.~Hofer, U.~Langer, I.~Toulopoulos, Discontinuous {Galerkin} isogeometric
  analysis on non-matching segmentation: Error estimates and efficient solvers,
  Tech. Rep. 2016-23, RICAM, Linz, Austria (2016).

\bibitem{dornisch:14}
W.~Dornisch, G.~Vitucci, S.~Klinkel, The weak substitution method -- an
  application of the mortar method for patch coupling in {NURBS}-based
  isogeometric analysis, Int. J. Numer. Methods Eng. 103~(3) (2015) 205--234.

\bibitem{hesch:12}
C.~Hesch, P.~Betsch, Isogeometric analysis and domain decomposition methods,
  {Comput. Methods Appl. Mech. Eng.} 213--216 (2012) 104--112.

\bibitem{coox:16}
L.~Coox, F.~Greco, O.~Atak, D.~Vandepitte, W.~Desmet, A robust patch coupling
  method for {NURBS}-based isogeometric analysis of non-conforming multipatch
  surfaces, Comput. Methods Appl. Mech. Eng. 316 (2017) 235--260.

\bibitem{horger:18}
T.~Horger, A.~Reali, B.~Wohlmuth, L.~Wunderlich, A hybrid isogeometric approach
  on multi-patches with applications to {Kirchhoff} plates and eigenvalue
  problems, in preparation.

\bibitem{marussig:17}
B.~Marussig, T.~J.~R. Hughes, A review of trimming in isogeometric analysis:
  Challenges, data exchange and simulation aspects, Arch. Comput. Methods Eng.
  (2017) 1--69.

\bibitem{delorenzis:14}
L.~De~Lorenzis, P.~Wriggers, T.~J.~R. Hughes, Isogeometric contact: a review,
  GAMM-Mitt. 37~(1) (2014) 85--123.

\bibitem{antolin:17}
P.~Antolin, A.~Buffa, M.~Fabre, A priori error for unilateral contact problems
  with {Lagrange} multiplier and isogeometric analysis,
  \texttt{https://arxiv.org/abs/1701.03150}.

\bibitem{ben_belgacem:99}
F.~{Ben Belgacem}, {The mortar finite element method with Lagrange
  multipliers}, Numer. Math. 84 (1999) 173--197.

\bibitem{bernardi:94}
C.~Bernardi, Y.~Maday, A.~T. Patera, A new nonconforming approach to domain
  decomposition: the mortar element method, in: H.~B. et.al. (Ed.), Nonlinear
  partial differrential equations and their applications., Vol.~XI, Coll\`{e}ge
  de France, 1994, pp. 13--51.

\bibitem{wohlmuth:01}
B.~Wohlmuth, Discretization Techniques and Iterative Solvers Based on Domain
  Decomposition, Vol.~17 of Lectures Notes in Computational Science and
  Engineering, Springer, Heidelberg, 2001.

\bibitem{brivadis:15}
E.~Brivadis, A.~Buffa, B.~Wohlmuth, L.~Wunderlich, Isogeometric mortar methods,
  Comput. Methods Appl. Mech. Eng. 284 (2015) 292--319.

\bibitem{wohlmuth:00}
B.~Wohlmuth, {A mortar finite element method using dual spaces for the Lagrange
  multiplier}, SIAM J. Numer. Anal. 38 (2000) 989--1012.

\bibitem{popp:09}
A.~Popp, M.~W. Gee, W.~A. Wall, A finite deformation mortar contact formulation
  using a primal-dual active set strategy, Int. J. Numer. Methods Eng. 79~(11)
  (2009) 1354--1391.

\bibitem{popp:10}
A.~Popp, M.~Gitterle, M.~W. Gee, W.~A. Wall, A dual mortar approach for 3{D}
  finite deformation contact with consistent linearization, Int. J. Numer.
  Methods Eng. 83~(11) (2010) 1428--1465.

\bibitem{popp:13}
A.~Popp, A.~Seitz, M.~W. Gee, W.~A. Wall, Improved robustness and consistency
  of 3{D} contact algorithms based on a dual mortar approach, Comput. Methods
  Appl. Mech. Eng. 264 (2013) 67--80.

\bibitem{popp:14}
A.~Popp, W.~A. Wall, Dual mortar methods for computational contact mechanics -
  overview and recent developments, GAMM-Mitt. 37~(1) (2014) 66--84.

\bibitem{lamichhane:07}
B.~Lamichhane, B.~Wohlmuth, Biorthogonal bases with local support and
  approximation properties, Math. Comp. 76~(257) (2007) 233--249.

\bibitem{wohlmuth:12a}
B.~Wohlmuth, A.~Popp, M.~W. Gee, W.~A. Wall, An abstract framework for a priori
  estimates for contact problems in 3{D} with quadratic finite elements,
  Comput. Mech. 49 (2012) 735--747.

\bibitem{popp:12}
A.~Popp, B.~Wohlmuth, M.~W. Gee, W.~A. Wall, Dual quadratic mortar finite
  element methods for 3{D} finite deformation contact, SIAM J. Sci. Comp.
  (2012) B421--B446.

\bibitem{seitz:16}
A.~Seitz, P.~Farah, J.~Kremheller, B.~Wohlmuth, W.~A. Wall, A.~Popp,
  Isogeometric dual mortar methods for computational contact mechanics, Comput.
  Methods Appl. Mech. Eng. 301 (2016) 259--280.

\bibitem{oswald:02}
P.~Oswald, B.~Wohlmuth, On polynominal reproduction of dual {FE} bases, in:
  N.~Debit, M.~Garbey, R.~Hoppe, D.~Keyes, Y.~Kuznetsov, J.~P\'eriaux (Eds.),
  Domain Decomposition Methods in Science and Engineering, CIMNE, 2002, pp.
  85--96, 13th International Conference on Domain Decomposition Methods, Lyon,
  France.

\bibitem{dornisch:17}
W.~Dornisch, J.~St\"ockler, R.~M\"uller, Dual and approximate dual basis
  functions for {B}-splines and {NURBS} – {Comparison} and application for an
  efficient coupling of patches with the isogeometric mortar method, Comput.
  Methods Appl. Mech. Eng. 316 (2017) 449 -- 496.

\bibitem{zou:17}
Z.~Zou, M.~Scott, M.~Borden, D.~Thomas, W.~Dornisch, E.~Brivadis, Isogeometric
  {B\'ezier} dual mortaring: Refineable higher-order spline dual bases and
  weakly continuous geometry, Comput. Methods Appl. Mech. Eng. 333 (2018) 497
  -- 534.

\bibitem{brivadis:15b}
E.~Brivadis, A.~Buffa, B.~Wohlmuth, L.~Wunderlich, The influence of quadrature
  errors on isogeometric mortar methods, in: B.~J{\"u}ttler, B.~Simeon (Eds.),
  Isogeometric Analysis and Applications 2014, Springer International
  Publishing, Cham, 2015, pp. 33--50.

\bibitem{wohlmuth:02}
Y.~{Maday}, F.~{Rapetti}, B.~{Wohlmuth}, {The influence of quadrature formulas
  in 2D and 3D mortar element methods.}, in: {Recent developments in domain
  decomposition methods. Some papers of the workshop on domain decomposition,
  ETH Z\"urich, Switzerland, June 7--8. 2001}, Springer, 2002, pp. 203--221.

\bibitem{maday:97}
L.~Cazabeau, C.~Lacour, Y.~Maday, Numerical quadrature and mortar methods, in:
  Computational Science for the 21st Century, John Wiley and Sons, 1997, pp.
  119--128.

\bibitem{farah:15}
P.~Farah, A.~Popp, W.~A. Wall, Segment-based vs. element-based integration for
  mortar methods in computational contact mechanics, Comput. Mech. 55~(1)
  (2015) 209--228.

\bibitem{brezzi:13}
D.~Boffi, F.~Brezzi, M.~Fortin, {Mixed Finite Element Methods and
  Applications}, Springer, Berlin, 2013.

\bibitem{benzi:05}
M.~Benzi, G.~H. Golub, J.~Liesen, Numerical solution of saddle point problems,
  Acta Numer. 14 (2005) 1--137.

\bibitem{baci}
W.~A. Wall, M.~Kronbichler, {BACI}: A multiphysics simulation environment,
  Tech. rep., Technische Universit\"at M\"unchen (2018).

\bibitem{muskhelishvili:77}
N.~I. Muskhelishvili, Some Basic Problems of the Mathematical Theory of
  Elasticity, Springer, 1977, translated from the Russian.

\bibitem{flemisch:04}
B.~Flemisch, M.~Puso, B.~Wohlmuth, A new dual mortar method for curved
  interfaces: 2{D} elasticity, Internat. J. Numer. Methods Eng. 63 (2005)
  813--832.

\bibitem{flemisch:07}
B.~Flemisch, B.~Wohlmuth, Stable {L}agrange multipliers for quadrilateral
  meshes of curved interfaces in 3{D}, Comput. Methods Appl. Mech. Eng. 196
  (2007) 1589--1602.

\end{thebibliography}
\end{document}